\documentclass[a4paper,11pt,reqno]{amsart}
\pdfoutput=1
\usepackage{epsfig,url,paralist}
\usepackage{amssymb}
\usepackage[normalem]{ulem}
\usepackage{fullpage,dynkin-diagrams}
\usepackage{amsmath}
\usepackage{tikz}
\usepackage{relsize}
\usepackage{fullpage}
\usepackage{hyperref,verbatim}
\usepackage[capitalize]{cleveref}
\usepackage{graphicx}
\usepackage{setspace,multirow}
\usepackage{enumitem,lineno}
\setlist{nolistsep}
\usepackage{lscape}
\usepackage{calrsfs}            
\usepackage{color}
\usepackage{colortbl}
\usepackage{xcolor}
\usepackage{wrapfig}
\usepackage[font=small]{caption}
\DeclareCaptionLabelFormat{tight}{#1~#2}
\usetikzlibrary{calc, intersections}

\numberwithin{equation}{section}
\newtheorem{theorem}{Theorem}[section]

\newtheorem{theorem*}{Main Result}
\newtheorem{lem}[theorem]{Lemma}
\newtheorem{coro}[theorem]{Corollary}
\newtheorem{cor*}[theorem*]{Corollary}
\newtheorem{prop}[theorem]{Proposition}
\theoremstyle{definition}
\newtheorem{defn}[theorem]{Definition}
\newtheorem{observation}[theorem]{Observation}

\newtheorem{nota}[theorem]{Notation}

\newtheorem{exmdefn}[theorem]{Example \& Definition}
\numberwithin{theorem}{section}
\usepackage{fancyhdr}
\def\<{\langle}
\def\>{\rangle}
\newcommand{\Res}{\mathsf{Res}}

\newcommand{\pperp}{\perp\hspace{-0.15cm}\perp}

\newcommand{\per}{\text{\,\footnotesize$\overline{\land}$\,}}
\newcommand{\proj}{\mathsf{proj}}

\newcommand{\Z}{\mathbb{Z}}

\newcommand{\cP}{\mathcal{P}}

\newcommand{\cL}{\mathcal{L}}

\DeclareMathOperator{\dist}{\mathsf{dist}}

\DeclareMathOperator{\Hdrei}{\mathsf{H}_3}
\DeclareMathOperator{\Hvier}{\mathsf{H}_4}

\begin{document}

\title{A point-line approach to the nonexistence of buildings of types $\Hdrei$ and $\Hvier$}



\date{February 19, 2025}

\begin{abstract}
We present a novel proof for the fact that thick spherical buildings of types $\mathsf{H_3}$ and $\mathsf{H_4}$ do not exist.
For that, we first provide an elementary axiom system for Lie incidence geometries associated with buildings of type $\mathsf{H_3}$.
This way, we can write our arguments purely in the language of point-line geometries, not needing the theory of buildings.
Assuming the existence of thick spherical buildings of type $\mathsf{H_3}$, 
we construct nontrivial root elations of generalized pentagons contained within them.
This leads to a contradiction with Tits's result on the nonexistence of Moufang generalized pentagons.
Consequently, we obtain a new, direct, and geometric proof for the nonexistence of thick spherical buildings of types $\mathsf{H_3}$ and $\mathsf{H_4}$,
without invoking Tits's extension theorem.
Together with similar geometric constructions of the author for root elations of buildings of types $\mathsf{B_n}$ and $\mathsf{C_n}$, and known constructions for  buildings of type $\mathsf{A_n}$, this yields an alternative, elementary proof for the fact that all thick irreducible spherical buildings of rank $3$ have the Moufang property, not using Tits's extension theorem.
\end{abstract}

\maketitle

\begin{center}
\begin{minipage}[c]{0.833\textwidth}
\begin{small}
\textbf{Author:} Sira Busch \\
\textbf{ORCID:} 0009-0009-0939-6543
\end{small}
\end{minipage}
\end{center}


\section{Introduction}

In 1976, an article by Jacques Tits was published, showing that Moufang $m$-gons only exist for $m \in \{ 2,3,4,6,8 \}$ (cf. \cite[Théorème~1]{Tits76}). 
Another result by Tits — proved in \cite{Tits77} and already mentioned in \cite[Addenda, p.~275]{SphericalTypeFiniteBNPairs} — states that thick spherical buildings of types $\mathsf{H_3}$ and $\mathsf{H_4}$ do not exist.
This latter result relies on the former and uses Tits’s extension theorem (\cite[Theorem~4.1.2]{SphericalTypeFiniteBNPairs}), which is regarded as rather technical (cf.~\cite[p.~7]{SphericalTypeFiniteBNPairs}).
Consequently, there have already been attempts to find different proofs avoiding the use of this extension theorem.

A new proof, that does not invoke the extension theorem, was given by Hendrik Van Maldeghem in 1995 (cf. \cite[\S~5]{Mal:95}). 
Van Maldeghem's argument assumes the existence of thick spherical buildings of type $\mathsf{H_3}$, 
shows that the rank-$2$ residues that are generalized pentagons are regular, and that regular generalized pentagons do not exist.

In this article, we present another alternative proof that is more direct and less intricate, as it does not require the notion of regularity.
The entire article will be written in the language of point-line geometries. 
Only in Section \ref{NoteOnBuildings} do we briefly comment on buildings; this section is not essential for our proof but provides additional context for interested readers.

For our main endeavor, we first recall some basic definitions in the context of point-line geometries and then introduce \emph{icosahedral geometries} as point-line geometries satisfying certain axioms (cf.~Definition \ref{icosahedralGeometriesDefn}).
We show that the maximal singular subspaces of icosahedral geometries are projective planes, that their apartments are icosahedra,
and prove that any two singular subspaces of an icosahedral geometry lie in a common apartment 
(it will then be clear that Lie incidence geometries of type $\mathsf{H_3}$ coincide with icosahedral geometries, as we will explain in Section \ref{NoteOnBuildings}).
Subsequently, we define root elations and the Moufang property for the point residues of an icosahedral geometry, which are generalized pentagons.

Assuming the existence of thick icosahedral geometries (or thick spherical buildings of type $\mathsf{H_3}$), 
we construct nontrivial root elations for the generalized pentagons they contain (cf.~Section \ref{constructionofelationsforH3}) 
and thereby obtain a contradiction to the nonexistence of Moufang generalized pentagons (cf.~Corollary \ref{constructionofelationsforH3corollary}).

Together with the geometric constructions for root elations of buildings of types $\mathsf{B_n}$ and $\mathsf{C_n}$ in \cite{Busch:TheMoufangConditionForPolarSpaces}, and the known constructions for buildings of type $\mathsf{A_n}$, 
this yields novel, geometric proofs for Tits's results on thick irreducible spherical buildings of rank~$3$ without the use of 
Tits’s extension theorem (or Tits's classification of all thick irreducible spherical buildings of rank at least three, cf.~Section \ref{NoteOnBuildings}).

Although the main motivation for this work lies in the theory of buildings, we have written it entirely in the language of point-line geometries to make it more accessible and potentially more broadly applicable.

\section{Icosahedral geometries}

We start with the most basic and essential definitions.

\begin{defn}
A \textit{point-line geometry} is a pair $\Delta=(\cP,\cL)$, where $\cP$ is a set and $\cL$ is a set of subsets of $\cP$ of cardinality at least two. The elements of $\cP$ are called \textit{points}, and the elements of $\cL$ are called \textit{lines}. If $p \in \cP$ and $L \in \cL$ with $p\in L$, we say that the point $p$ \emph{lies on} the line $L$, and the line $L$ \emph{contains} the point $p$, or \emph{goes through} $p$. If two points $p$ and $q$ lie on a common line, they are called \textit{collinear}, denoted $p \perp q$. For any point $p$ and any subset $P \subseteq \cP$, we define \[p^\perp := \{q \in \cP\mid q \perp p\} \text{ and } P^\perp := \bigcap_{p \in P} p^\perp.\]
Note that, if there exists a line $L$ through a point $p$ in $\Delta$, then $p$ is on a common line with itself, and therefore $p \in p^{\perp}$.
\end{defn}

\begin{defn}
Let $\Delta = (\cP,\cL)$ be a point-line geometry.  
\begin{enumerate}[label=(\roman*)]
\item A subset $S$ of $\cP$ is called a \textit{subspace} of $\Delta$ if every line $L\in\cL$ that contains at least two points of $S$ is contained in $S$. 
We usually regard subspaces of $\Delta$ as subgeometries of $\Delta$ in the canonical way.
\item A subspace $S$ in which all points are collinear, or equivalently, for which $S \subseteq S^\perp$, is called a \textit{singular subspace}. 
If, moreover, $S$ is not contained in any other singular subspace, it is called a \textit{maximal singular subspace}. 
A singular subspace is called \textit{projective} if, as a subgeometry, it is a projective space. 				
\item For a subset $P$ of $\cP$, the \textit{subspace generated by $P$} is denoted $\<P\>_\Delta = \<P\>$ and is defined to be the intersection of all subspaces containing $P$. 
A subspace generated by three mutually collinear points not on a common line is called a \textit{plane}. (Note that, in general, this is not necessarily a singular subspace.) 
\end{enumerate}
\end{defn}

\begin{exmdefn}
A \emph{projective plane} is a point-line geometry $(\mathcal{P}, \mathcal{L})$ such that (cf.\ \cite[Ch.\,2, p.\,34]{Handbook})
\begin{enumerate}[label=(\roman*)]
\item for $p, q \in \mathcal{P}$, there always exists a unique $L \in \mathcal{L}$ such that $p,q \in L$. (Any two points are collinear.)
\item for $L, J \in \mathcal{L}$, there always exists a unique $p \in \mathcal{P}$ such that $p \in L \cap J$. (Any two lines intersect.)
\item there exist at least $4$ points of which any $3$ points are non-collinear. (There exists a quadrangle.)
\end{enumerate}
\end{exmdefn}

\begin{exmdefn}
A point-line geometry $(\mathcal{P}, \mathcal{L})$ is called a \emph{projective space} if we are in one of the following scenarios:
\begin{enumerate}[ 
  leftmargin=4em,   
  rightmargin=1em,   
]
	\item $\mathcal{P}$ is the empty set.
	\item $\mathcal{P}$ consists of one point.
	\item $(\mathcal{P}, \mathcal{L})$ consists of one line with at least $3$ points.
	\item $(\mathcal{P}, \mathcal{L})$ is a projective plane.
	\item $(\mathcal{P}, \mathcal{L})$ has the following properties:
		\begin{enumerate}[label=(\roman*)] 
		\item For $p, q \in \mathcal{P}$, there always exists a unique $L \in \mathcal{L}$ such that $p,q \in L$.
		\item If $a,b,c,d$ are different points and $ab$ and $cd$ intersect, then $ad$ and $bc$ intersect.
		\item There exist at least $4$ points of which any $3$ points are non-collinear.
		\item Every line has at least three points.
		\item Any chain of subspaces has finite length. 
		\end{enumerate}
\end{enumerate}
The characterization of (5) is due to Veblen and Young. Axiom (ii) of (5) is often called the \emph{Veblen-Young axiom} or \emph{Pasch's axiom} (cf.\ \cite[Proposition 2.2.8]{MonographPolarSpaces}. 
\end{exmdefn}

\begin{defn} Let $\Delta = (\cP,\cL)$ be a point-line geometry.  
\begin{enumerate}[label=(\roman*)]
\item We can construct a graph for $\Delta$ by drawing a vertex for every point $p \in \mathcal{P}$ and drawing an edge between two points $p$ and $q$ if $p$ and $q$ are collinear. This graph is called the \emph{point graph} or \emph{collinearity graph} of $\Delta$.
\item The \emph{distance between two points} $p$ and $q$ in $\Delta$ is the graph distance between the corresponding vertices in the collinearity graph.
\item We define the \emph{diameter of $\Delta$} to be the maximal distance that two points in $\Delta$ can have, and $\infty$ if such a maximum does not exist.
\end{enumerate}
\end{defn}

\begin{nota}
The geometries we will mostly be interested in have diameter $3$. For point-line geometries of diameter $3$, we introduce the following notation.
\begin{enumerate}[label=(\roman*)]
\item If two points $p$ and $b$ have distance $2$, then we write $p \pperp b$. 
\item We denote by $p^{\pperp}$ the set of all points of $\Delta$ that are at distance $2$ from $p$.
\item We call two points $p$ and $b$ \emph{opposite} if they have distance $3$. 
\end{enumerate}
\end{nota}

\begin{exmdefn} A \emph{generalized pentagon} is a point-line geometry $\Gamma = (\cP, \cL)$ such that the following axioms are satisfied (cf. \cite[Definition 1.3.1]{GeneralizedPolygons}).
\begin{enumerate}[label=(\roman*)]
\item $\Gamma$ contains no ordinary $k$-gon as a subgeometry for $2 \leq k < 5$.
\item Any two elements of $\cP \cup \cL$ are contained in some ordinary pentagon in $\Gamma$, and these are also called the \emph{apartments of $\Gamma$}.
\item There exists an ordinary hexagon as a subgeometry in $\Gamma$.
\end{enumerate}

An equivalent definition is the following (cf. \cite[Lemma 1.3.6]{GeneralizedPolygons}).

A point-line geometry $\Gamma$ is a generalized pentagon if and only if the collinearity graph of $\Gamma$ is a connected graph of diameter $2$ and girth $5$ such that each vertex lies on at least two edges.
If every vertex lies on at least three edges, then $\Gamma$ is thick.
If every vertex lies on exactly two edges, then $\Gamma$ is called \emph{thin}.
\end{exmdefn}

We can now define \emph{icosahedral geometries}.

\begin{defn}[\textbf{Icosahedral geometries}]\label{icosahedralGeometriesDefn}
An \emph{icosahedral geometry} $\Delta$ is a point-line geometry such that the following axioms are satisfied.
\begin{enumerate}
\item[(I1)] $\Delta$ has diameter $3$, and two distinct points $p$ and $q$ can have one of the following relations.
	\begin{enumerate}[label=(\roman*)]
	\item $p$ and $q$ are collinear, and there is a unique line containing $p$ and $q$, which is denoted by $pq$,
	\item $p$ and $q$ have distance $2$, and $p^{\perp} \cap q^{\perp}$ is a line,
	\item $p$ and $q$ are opposite, and $p^{\perp} \cap q^{\pperp}$ and $p^{\pperp} \cap q^{\perp}$ are generalized pentagons.
	\end{enumerate}
\item[(I2)] For every point $p$, there exists a point $q$ which is opposite $p$.
\item[(I3)] For a point $p$ and a line $L$ not containing $p$, we have that
	\begin{enumerate}[label=(\roman*)]
	\item if $p$ is collinear with $2$ points of $L$, then $p$ is collinear with all points of $L$,
	\item if $p$ has distance $2$ to $2$ points of $L$, then $p$ has distance at most $2$ to all points of $L$,
	\item if $p$ has distance $2$ to all points of $L$, then there exists a unique point $q$ such that $p \perp q \perp L$,
	\item if $p$ is opposite some point of $L$, then there exists exactly one point on $L$ at distance $2$ from $p$.
	\end{enumerate}
\end{enumerate}
\end{defn}

Our next goal is to show that the planes of icosahedral geometries in which every line contains at least three points are projective planes.
In the following, $\Delta$ always denotes an icosahedral geometry. 

\begin{lem}\label{lemIcosahedralGeom01}
For every line $L$, there exist two points $p$ and $q$ such that $L=p^{\perp} \cap q^{\perp}$.
\end{lem}

\begin{proof}
Let $L$ be a line, let $y$ be some point on $L$, and let $z$ be a point at distance $3$ from $y$. 
Then $z^{\pperp} \cap y^{\perp}$ is a generalized pentagon. 
Since $z$ cannot have distance $3$ to all points of $L$ and $L \subset y^{\perp}$, there exists some point $a$ on $L$
contained in $z^{\pperp} \cap y^{\perp}$. Let $M$ and $N$ be two lines in $z^{\pperp} \cap y^{\perp}$ through $a$.
Then there exist two points $m$ and $n$ on $M$ and $N$, respectively, which are not collinear. 
Since $m$ and $n$ are both collinear to $y$ and $a$, they are collinear to all of $L$, and thus $L=m^{\perp} \cap n^{\perp}$.
\end{proof}

\begin{lem}\label{lemIcosahedralGeom02}
For every pair of different points $a, b \in \cP$, there exists a point $p \in \cP$ that is collinear to $a$ but not to $b$.
\end{lem}

\begin{proof}
\textbf{Case 1:} $\dist(a, b)=3$.
Take any point in $a^{\perp} \cap b^{\pperp}$.

\textbf{Case 2:} $\dist(a, b)=2$.
Let $L$ be the line $a^{\perp} \cap b^{\perp}$. 
Let $x$ be a point at distance $3$ from $a$. 
Then we can take any point in the generalized pentagon $a^{\perp} \cap x^{\pperp}$, except the points of $L$, if $L \subset a^{\perp} \cap x^{\pperp}$.

\textbf{Case 3:} $a \perp b$.
Let $x$ be a point at distance $3$ from $a$. Then there exists some point $c$ on $ab$ at distance $2$ from $x$.
Consider the generalized pentagon $x^{\pperp} \cap a^{\perp}$. 
If there did not exist a point not collinear to $b$ in $x^{\pperp} \cap a^{\perp}$, then every point in $x^{\pperp} \cap a^{\perp}$ would be collinear to the whole line $ac$. This contradicts the axioms of a generalized pentagon.
\end{proof}

\begin{lem}\label{abandxyintersect}
Let $p$ and $q$ be two opposite points.
Let $a$ and $b$ be two collinear points in $\Gamma := p^{\perp} \cap q^{\pperp}$.
Let $x$ be a point on $pa$, and let $y$ be a point on $pb$ such that $x \neq y$, and $xy \neq ab$.
Then $ab$ and $xy$ intersect in a unique point.
\end{lem}

\begin{proof}
If $xy$ conincides with $ap$ or $pb$, then $ab \cap xy$ is either the unique point $a$ or the unique point $b$.
In the following we may assume that $x$ is dinstinct from both $p$ and $a$, and $y$ is distinct from both $p$ and $b$.

The point $q$ has distance $2$ to all points of $ab$. Thus, there exists a unique point $c$ such that $q \perp c \perp ab$.
Suppose the point $c$ had distance $2$ to all points of $xy$.
Then there would have to exist a unique point $v$ such that $c \perp v \perp xy$.
But every point on $ab$ is collinear to both $c$ and every point of $xy$. 
It follows that the assumption was wrong, and $c$ has to be collinear to a point of $xy$.
But since $p^\perp \cap c^\perp$ is exactly $ab$, this point has to be on $ab$.
Thus, $xy$ and $ab$ intersect in a point, which is unique because of axiom (I1)(i).
\end{proof}

\begin{lem}\label{findpointoppoanddist2toline}
Let $p$ and $q$ be two opposite points.
Let $a$ and $b$ be two collinear points in $p^{\perp} \cap q^{\pperp}$.
Let $x$ be a point on $pa$ distinct from $p$, and let $y$ be a point on $pb$ distinct from $p$.
Then there exists a point $q'$ opposite $p$ such that $q'$ has distance $2$ to every point on $xy$.
\end{lem}

\begin{proof}
\textbf{(i)} If every line in $\Delta$ contains only two points, then $x=a$ and $y=b$, and we are done.

\textbf{(ii)} Now suppose $x=a$, but $y \neq b$.
Then $q$ has distance $3$ to $y$, and we put
$$\Gamma_{q,y} := q^\perp \cap y^{\pperp}, \text{ and } \Gamma_{q,p}:=q^\perp \cap p^{\pperp}.$$
Let $c$ be the unique point collinear to $q$ and $ab$.
Let $d$ be a point on $a^\perp \cap q^\perp$ distinct from $c$.
Then $d$ is contained in $\Gamma_{q,y} \cup \Gamma_{q,p}$.
Let $K$ be a line through $d$ in $\Gamma_{q,y}$ not equal to $cd$, 
and let $q'$ be a point on $K$ not equal to $d$.
Then $q'$ has distance $2$ to $a$ and $y$.
Suppose $q'$ were not opposite $p$.
Then $q'$ would have distance at most two to all points on $pb$.

a) Suppose $q'$ were collinear to some point $q''$ on $pb$.
Since $q \perp q'$, we would have that $q'' \in q^{\pperp} \cap p^\perp$.
But if $q'' \in q^{\pperp} \cap p^\perp$ and $q'' \in pb$, then $q''=b$.
It follows that $q' \in b^\perp \cap q^\perp$.
That means the line $K$ through $d$ intersects $b^\perp \cap q^\perp$ in $q'$.

Now either $q'=c$, but then $K=cd$, which contradicts the fact that $K$ was chosen different from $cd$, or $q' \neq c$.
But if $q' \neq c$, then, $d$ is collinear to $c$ and $q'$ and thus to all points of $b^\perp \cap q^\perp$.
Since all of the points $d$, $c$, and $q'$ are contained in $\Gamma_{q,p}$, this yields a contradiction to the fact that $\Gamma_{q,p}$ is a generalized pentagon.
It follows that $q'$ is opposite $p$.

b) Suppose $q'$ had distance $2$ to all points of $pb$. 
Then $b$ would have distance $2$ to all points of $q'd$, and there would have to exist a unique point $e$ collinear to $b$ and $q'd$.
We have that $b^{\perp} \cap d^{\perp} = ca$. Thus, $e$ has to be on $ca$.
If $e \neq c$, then $e$ and $q$ have distance $2$ and $e^\perp \cap q^\perp = cd$ would contain $q'$.
But $q'$ can not be on $cd$ since $K \neq cd$.
If $e=c$, then $q'$ is collinear to $c$ and $d$ and there would be a triangle in $\Gamma_{q,p}$, contradicting the fact that it is a generalized pentagon.

Thus, $q'$ is a point opposite $p$ and at distance $2$ to all points of $xy$ (cf.~Figure \ref{Lemma210}).

\begin{center}
    \resizebox{2.2cm}{!}{
        \begin{tikzpicture}[baseline={(0, 4.2)}, scale=0.5]
            
            \def\pentagonradius{3}
            \def\angle{72}
            \def\rotationP{-90}
            \def\rotationB{-54}
            \def\verticalshift{8}

            \begin{scope}[scale=1, yscale=0.6]
                \foreach \i in {0,1,2,3,4} {
                    \coordinate (P\i) at (\i*\angle+\rotationP:\pentagonradius);
                }
                \draw[thick, opacity=0.3] (P4) -- (P0) -- (P1);

                \coordinate (P) at (0,-5);
            \end{scope}

            \begin{scope}[scale=1, yscale=0.6, shift={(0,\verticalshift)}]
                \foreach \i in {0,1,2,3,4} {
                    \coordinate (B\i) at (\i*\angle+\rotationB:\pentagonradius);
                }
                \draw[thick, opacity=0.3] (B3) -- (B4) -- (B0);

                \coordinate (B) at (0,5);
            \end{scope}

            \draw[thick, opacity=0.3] (P0) -- (B0) -- (P1);
            \draw[thick, opacity=0.3] (B3) -- (P4) -- (B4) -- (P0);

            \draw[thick, opacity=0.3] (P0) -- (P);
            \draw[thick, opacity=0.3] (P1) -- (P);
            \draw[thick, opacity=0.3] (P4) -- (P);

            \draw[thick] (B0) -- (B);
            \draw[thick, opacity=0.3] (B3) -- (B);
            \draw[thick] (B4) -- (B);

            \draw[thick] (P0) -- (P) -- (P1) -- cycle;

            \coordinate (y)  at ($ (P) + 0.5*(P1) - 0.5*(P) $);
            \coordinate (q') at ($ (B) + 0.5*(B3) - 0.5*(B) $);

            \draw[thick] (B) -- (q') -- (B4) -- (B0);
            \draw[thick] (y) -- (P0);

            \node at (B0) [circle, fill, inner sep=1pt, label={[label distance=0mm, yshift=-1mm]above left:{\Large \( c \)}}] {};
            \node at (B4) [circle, fill, inner sep=1pt, label={[label distance=0mm, yshift=-1mm]above right:{\Large \( d \)}}] {};
            \node at (B)  [circle, fill, inner sep=1pt, label={above:{\Large \( q \)}}] {};
            \node at (P)  [circle, fill, inner sep=1pt, label={below:{\Large \( p \)}}] {};
            \node at (P0) [circle, fill, inner sep=1pt, label={[label distance=0mm]above:{\Large \( a \)}}] {};
            \node at (P1) [circle, fill, inner sep=1pt, label={[label distance=-1mm]right:{\Large \( b \)}}] {};
            \node at (y)  [circle, fill, inner sep=1pt, label={[label distance=-1mm]below right:{\Large \( y \)}}] {};
            \node at (q') [circle, fill, inner sep=1pt, label={[label distance=-1mm]above left:{\Large \( q' \)}}] {};

        \end{tikzpicture}
    }
    \captionof{figure}{}\label{Lemma210}
\end{center}

\textbf{(iii)} Now suppose $x \neq a$ and $y \neq b$.
According to Lemma \ref{abandxyintersect}, $ab$ and $xy$ intersect. Let $z$ be their intersection point.
Then we can apply the steps of (ii) again, interchanging $a$ and $z$.
\end{proof}

\begin{lem}\label{everyLineInPentagon}
Let $L$ be a line in $\Delta$. Then there exist two opposite points $p$ and $b$ such that $L \subset p^\perp \cap b^{\pperp}$.
In particular, every line $L$ in $\Delta$ is contained in a generalized pentagon.
\end{lem}

\begin{proof}
Let $L$ be a line in $\Delta$.
According to Lemma \ref{lemIcosahedralGeom01}, there exist two points $a$ and $b$ such that $a^\perp \cap b^\perp = L$. 
Let $p$ be a point opposite $b$. 
Then for every point $\ell_i$ on $L$, $p$ has distance $2$ to a unique point $p_i$ on $b\ell_i$.
If these points $p_i$ are not the points $\ell_i$ on $L$, then, according to Lemma \ref{findpointoppoanddist2toline}, we can find a point $p'$
such that $p'^{\pperp} \cap b^\perp$ contains $L$.
\end{proof}

\begin{lem}
Let $b$, $c$, $d$, $e$ be four different points that are pairwise collinear,
but no three of which are on a common line. 
If $bc \cap de$ is a unique point $a$, then $bd \cap ce$ is also a unique point.
\end{lem}

\begin{proof} 
Let $p$ be a point opposite $a$. Then $p$ has distance $2$ to a unique point $p_1$ on $bc$ and to a unique point $p_2$ on $de$. 
Since $b$ is collinear to $a$ and $e$, $b$ is collinear to $p_2$. Then $p_2$ is collinear to $b$ and $a$ and thus to $p_1$.
With that, $p$ has distance $2$ to all points of the line $p_1p_2$. 
That means there exists a unique point $q$ such that $p \perp q \perp p_1p_2$.
Since $p$ and $a$ have distance $3$, $q$ and $a$ have distance $2$, and $q^\perp \cap a^\perp$ is exactly $p_1p_2$.

Suppose $p_1=b$ and $p_2=d$, and suppose $bd \cap ce = \varnothing$. 
Then $q$ had distance $2$ to all points of $ce$, and there has to exist a unique point $y$ such that $q \perp y \perp ce$.
But that contradicts the fact that any point on $bd$ is collinear to $q$ and any point of $ce$.
Thus, $q$ has to be collinear to a point of $ce$. 
But since $q^\perp \cap a^\perp$ is exactly the line $bd$, this point has to be on $ce$, 
and it follows that $bd$ and $ce$ intersect in a unique point.

The same argument goes through if $p_1=c$ and $p_2=e$.

Now suppose $p_1 \neq b,c$ and $p_2 \neq d,e$.

According to Lemma \ref{findpointoppoanddist2toline}, we can find a point $p'$ opposite $a$ such that $p'$ has distance $2$ to all points of $bd$.
That means $bd$ is a line of the generalized pentagon $a^\perp \cap p'^{\pperp}$.
Since $c$ is a point of $ba$, $e$ is a point of $da$, and $c \perp e$, the line $ce$ has to intersect $bd$ in a unique point according to Lemma \ref{abandxyintersect}.
\end{proof}

\begin{coro}\label{thicklinesprojectiveplanes}
If every line in $\Delta$ contains at least three points, then the maximal singular subspaces of $\Delta$ are projective planes. 
\end{coro}


\begin{observation}\label{everypointdisttwotolineofplane}
Let $\Delta$ be an icosahedral geometry in which every line contains at least three points.
Let $p$ be a point, let $q$ be a point opposite $p$, and let $\pi$ be a projective plane through $q$. 
Since $p$ has distance $2$ to at least one point on every line through $q$, $p^{\pperp} \cap \pi$ is a line.
\end{observation}

\begin{defn}
Let $\Delta$ be an icosahedral geometry in which every line contains at least three points.
Let $p$ be a point in $\Delta$. Then we can define a new point-line geometry $\Res_{\Delta}(p)$ as follows.
\begin{enumerate}[label=(\roman*)]
\item The set of points of $\Res_{\Delta}(p)$ is given by all lines through $p$.
\item The set of lines of $\Res_{\Delta}(p)$ is given by all projective planes in $\Delta$.
\item Two points of $\Res_{\Delta}(p)$ lie on a common line if the corresponding lines in $\Delta$ lie in a common projective plane.
\end{enumerate}
We call $\Res_{\Delta}(p)$ the \emph{point residue of $p$ in $\Delta$}.
\end{defn}

\begin{observation}
Let $\Delta$ be an icosahedral geometry in which every line contains at least three points.
For a point $p$ and a point $q$ opposite $p$ in $\Delta$, we have
$\Res(p) \cong p^{\perp} \cap q^{\pperp}$.
This is the case since $q$ has distance $2$ to a unique point of every line through $p$.
In particular, this means that all point residues are generalized pentagons.
\end{observation}

\begin{defn}
Let $\Delta$ be an icosahedral geometry. Then we call $\Delta$ 
\begin{enumerate}[label=(\roman*)]
\item \emph{thick} if every line is contained in at least $3$ planes.
\item \emph{thin} if every line is contained in exactly $2$ planes.
\end{enumerate}
\end{defn}

\begin{lem}
If $\Delta$ is thick, then every line has at least $3$ points, and every point residue is a thick generalized pentagon.
\end{lem}

\begin{proof}
Suppose $\Delta$ is thick. Let $L$ be a line in $\Delta$, and let $a$ be a point on $L$.
Let $p$ be a point opposite $a$ and at distance $2$ to some point $b$ on $L$.

Since $\Delta$ is thick, there exist at least three (non-projective) planes through $L$.
Let $K_1$, $K_2$, and $K_3$ be three lines through $a$ not equal to $L$.
The point $p$ must have distance $2$ to $b$ and to a point $c_i$ on each of the lines $K_i$, for $i \in \{1,2,3\}$.
Furthermore, $a^\perp \cap p^{\pperp}$ is a generalized pentagon according to (I1)(iii).

In $a^\perp \cap p^{\pperp}$, $b$ is contained in the $3$ lines $bc_1$, $bc_2$, and $bc_3$.
Now let $x$ be any other point in $a^\perp \cap p^{\pperp}$.
Since $\Delta$ is thick, the line $ax$ is contained in three planes, and again, $p$ has to have
distance $2$ to a unique point of any of the lines through $a$ in these planes.
It follows that $x$ lies on three edges in $a^\perp \cap p^{\pperp}$, and, since $x$ was arbitrary,
$a^\perp \cap p^{\pperp}$ is a thick generalized pentagon.
Hence, every line in $a^\perp \cap p^{\pperp}$ contains at least $3$ points.
\end{proof}

\section{Apartments}

Let $\Delta$ be a thick icosahedral geometry. 

\begin{defn}
Let $\alpha$ be a projective plane in $\Delta$.
Let $P$ be a set of three points of $\alpha$ not contained in a common line.
Let $$L := \{ (p_i, p_j) \mid p_i, p_j \in P,~ i \neq j\}.$$
We define a \emph{triangle of $\alpha$} as the point-line geometry given by $(P,L)$.
Note that any such set $P$ generates $\alpha$ in $\Delta$.
\end{defn}

\begin{defn}\label{DefnApartment}
Let $\Delta$ be an icosahedral geometry, and let $p$ and $q$ be two opposite points.  
Let $\Gamma_q$ be an apartment of $q^\perp \cap p^{\pperp}$.
Then, by projecting every point and every line of $\Gamma_q$ to the point $p$, we obtain an apartment $\Gamma_p$ of $p^\perp \cap q^{\perp}$.
We call the subgeometry consisting of 
\begin{enumerate}[label=(\roman*)]
\item the points and lines of $\Gamma_p$ and $\Gamma_q$, 
\item the points $p$ and $q$, 
\item the lines between $p$ and points of $\Gamma_p$,
\item the lines between $q$ and points of $\Gamma_q$,
\item the lines between points of $\Gamma_p$ and $\Gamma_q$,
\item and the triangles appearing naturally,
\end{enumerate}
an \emph{apartment of $\Delta$}. 
\end{defn}

\begin{observation}
The apartments of an icosahedral geometry are icosahedra.
Every apartment is a thin icosahedral geometry with exactly two points on each line.
\end{observation}

\begin{observation}\label{everyPlaneInApartment}
Let $\alpha$ be a plane in $\Delta$. Then there exist two opposite points $p$ and $b$ 
such that $\alpha$ contains $p$ and intersects $p^\perp \cap b^{\pperp}$ in a line.
With Lemma \ref{findpointoppoanddist2toline}, it follows that any triangle of any plane in $\Delta$ is contained in an apartment.
\end{observation}

\begin{nota}
By abuse of language, we will also say that \emph{the plane $\alpha$ is contained in an apartment $\Sigma$} if some triangle of $\alpha$ is contained in $\Sigma$.
Similarly, we will say that a line is contained in $\Sigma$ if two distinct points of $L$ are contained in $\Sigma$.
\end{nota}

\begin{lem}\label{canfindapartmentthroughoppopoints}
Given a point $p$ and a line $L$ contained in $p^{\perp}$, 
we can always find an apartment containing $p$ and $L$.
\end{lem}

\begin{proof}
If $L$ does not contain $p$, we can use Lemma \ref{everyLineInPentagon} to find a point $q$ opposite $p$ such that $L \subset p^{\perp} \cap q^{\pperp}$.
Then we choose an apartment $\Gamma_p$ of $p^{\perp} \cap q^{\pperp}$ containing $L$ and proceed as in Definition \ref{DefnApartment}.

If $L$ contains $p$, then let $q$ be any point opposite $p$. Let $a$ be the unique point of $L$ at distance $2$ to $q$. 
Then we choose an apartment $\Gamma_p$ of $p^{\perp} \cap q^{\pperp}$ containing $a$ and proceed as in Definition \ref{DefnApartment}.
\end{proof}

\begin{defn}
We say two planes $\alpha$ and $\beta$ in $\Delta$ with $\alpha \cap \beta = \varnothing$ are
\begin{enumerate}[label=(\roman*)]
\item \emph{opposite} if every point of $\alpha$ has distance $2$ to a unique line of $\beta$, and distance $3$ to all other points of $\beta$,
\item \emph{far} if there exists a unique point $a$ of $\alpha$ collinear with a unique point $b$ of $\beta$, all other points of $\alpha$ have distance $2$ to a unique line of $\beta$ through $b$, and are opposite all other points of $\beta$,
\item \emph{close} if there exists a unique point $a$ in $\alpha$ collinear with a unique line $L$ of $\beta$, 
and there exists a unique line $K$ of $\alpha$ through $a$ collinear with a unique point $b$ of $L$ in $\beta$.
In that case, the points in $\alpha$ that are not on $K$ have distance $2$ to the points of $L$ and distance $3$ to the points in $\beta \setminus L$.
\end{enumerate}
\end{defn}

\begin{center}
\begin{tabular}{>{\centering\arraybackslash} m{3.1cm}  >{\centering\arraybackslash} m{3.1cm} >{\centering\arraybackslash} m{3.1cm}}

\resizebox{2.2cm}{!}{
\begin{tikzpicture}[baseline={(0, 4.2)}, scale=0.5]
    \def\pentagonradius{3}
    \def\angle{72}
    \def\rotationP{-90}
    \def\rotationB{-54}
    \def\verticalshift{8}

    \begin{scope}[scale=1, yscale=0.6]
        \foreach \i in {0,1,2,3,4} {
            \coordinate (P\i) at (\i*\angle+\rotationP:\pentagonradius);
        }
        \draw[thick] (P0) -- (P1) -- (P2) -- (P3) -- (P4) -- cycle;

        \draw[thick] (P0) -- (P1);
        \draw[thick] (P1) -- (P2);
        \draw[thick] (P2) -- (P3);
        \draw[thick] (P3) -- (P4);
        \draw[thick] (P4) -- (P0);

        \coordinate (P) at (0,-5);
    \end{scope}

    \begin{scope}[scale=1, yscale=0.6, shift={(0,\verticalshift)}]
        \foreach \i in {0,1,2,3,4} {
            \coordinate (B\i) at (\i*\angle+\rotationB:\pentagonradius);
        }
        \draw[thick] (B0) -- (B1) -- (B2) -- (B3) -- (B4) -- cycle;

        \draw[thick] (B0) -- (B1);
        \draw[thick] (B1) -- (B2);
        \draw[thick] (B2) -- (B3);
        \draw[thick] (B3) -- (B4);
        \draw[thick] (B4) -- (B0);

        \coordinate (B) at (0,5);
    \end{scope}

    \draw[thick, opacity=0.3] (P0) -- (B0) -- (P1) -- (B1) -- (P2) -- (B2) -- (P3) -- (B3) -- (P4) -- (B4) -- (P0);

    \draw[thick, opacity=0.3] (P0) -- (P);
    \draw[thick, opacity=0.3] (P1) -- (P);
    \draw[thick, opacity=0.3] (P2) -- (P);
    \draw[thick, opacity=0.3] (P3) -- (P);
    \draw[thick, opacity=0.3] (P4) -- (P);

    \draw[thick, opacity=0.3] (B0) -- (B);
    \draw[thick, opacity=0.3] (B1) -- (B);
    \draw[thick, opacity=0.3] (B2) -- (B);
    \draw[thick, opacity=0.3] (B3) -- (B);
    \draw[thick, opacity=0.3] (B4) -- (B);

    \draw[color=black, line width=1pt] (B) -- (B1) -- (B2) -- cycle;
    \draw[color=black, line width=1pt] (P) -- (P4) -- (P0) -- cycle;
    \fill[color=black, opacity=0.3] (B) -- (B2) -- (B1) -- cycle;
    \fill[color=black, opacity=0.3] (P) -- (P0) -- (P4) -- cycle;

    \node at (B0) [circle, fill, inner sep=1pt, label={[label distance=0mm]above left:{\Large \(  \)}}] {};
    \node at (B1) [circle, fill, inner sep=1pt, label={[label distance=0mm]right:{\Large \(  \)}}] {};
    \node at (B2) [circle, fill, inner sep=1pt, label={[label distance=-1mm]above left:{\Large \(  \)}}] {};
    \node at (B3) [circle, fill, inner sep=1pt, label={[label distance=0mm]left:{\Large \(  \)}}] {};
    \node at (B4) [circle, fill, inner sep=1pt, label={[label distance=-1mm] above right:{\Large \(  \)}}] {};
    \node at (B) [circle, fill, inner sep=1pt, label={above:{\Large \(  \)}}] {};
    \node at (P) [circle, fill, inner sep=1pt, label={below:{\Large \(  \)}}] {};
    \node at (P0) [circle, fill, inner sep=1pt, label={[label distance=-1mm]below left:{\Large \(  \)}}] {};
    \node at (P1) [circle, fill, inner sep=1pt, label={[label distance=-1mm]below:{\Large \(  \)}}] {};
    \node at (P2) [circle, fill, inner sep=1pt, label={[label distance=-1mm]above left:{\Large \(  \)}}] {};
    \node at (P3) [circle, fill, inner sep=1pt, label={[label distance=-1mm]above left:{\Large \(  \)}}] {};
    \node at (P4) [circle, fill, inner sep=1pt, label={[label distance=-1mm]below left:{\Large \(  \)}}] {};
\end{tikzpicture}}
\captionof{figure}{\\ Opposite planes}
\label{planesicoopposite}

&

\resizebox{2.2cm}{!}{
\begin{tikzpicture}[baseline={(0, 4.2)}, scale=0.5]
    \def\pentagonradius{3}
    \def\angle{72}
    \def\rotationP{-90}
    \def\rotationB{-54}
    \def\verticalshift{8}

    \begin{scope}[scale=1, yscale=0.6]
        \foreach \i in {0,1,2,3,4} {
            \coordinate (P\i) at (\i*\angle+\rotationP:\pentagonradius);
        }
        \draw[thick] (P0) -- (P1) -- (P2) -- (P3) -- (P4) -- cycle;

        \draw[thick] (P0) -- (P1);
        \draw[thick] (P1) -- (P2);
        \draw[thick] (P2) -- (P3);
        \draw[thick] (P3) -- (P4);
        \draw[thick] (P4) -- (P0);

        \coordinate (P) at (0,-5);
    \end{scope}

    \begin{scope}[scale=1, yscale=0.6, shift={(0,\verticalshift)}]
        \foreach \i in {0,1,2,3,4} {
            \coordinate (B\i) at (\i*\angle+\rotationB:\pentagonradius);
        }
        \draw[thick] (B0) -- (B1) -- (B2) -- (B3) -- (B4) -- cycle;

        \draw[thick] (B0) -- (B1);
        \draw[thick] (B1) -- (B2);
        \draw[thick] (B2) -- (B3);
        \draw[thick] (B3) -- (B4);
        \draw[thick] (B4) -- (B0);

        \coordinate (B) at (0,5);
    \end{scope}

    \draw[thick, opacity=0.3] (P0) -- (B0) -- (P1) -- (B1) -- (P2) -- (B2) -- (P3) -- (B3) -- (P4) -- (B4) -- (P0);

    \draw[thick, opacity=0.3] (P0) -- (P);
    \draw[thick, opacity=0.3] (P1) -- (P);
    \draw[thick, opacity=0.3] (P2) -- (P);
    \draw[thick, opacity=0.3] (P3) -- (P);
    \draw[thick, opacity=0.3] (P4) -- (P);

    \draw[thick, opacity=0.3] (B0) -- (B);
    \draw[thick, opacity=0.3] (B1) -- (B);
    \draw[thick, opacity=0.3] (B2) -- (B);
    \draw[thick, opacity=0.3] (B3) -- (B);
    \draw[thick, opacity=0.3] (B4) -- (B);

    \draw[color=black, line width=1pt] (B) -- (B1) -- (B0) -- cycle;
    \draw[color=black, line width=1pt] (P) -- (P4) -- (P0) -- cycle;
    \fill[color=black, opacity=0.3] (B) -- (B0) -- (B1) -- cycle;
    \fill[color=black, opacity=0.3] (P) -- (P0) -- (P4) -- cycle;

    \node at (B0) [circle, fill, inner sep=1pt, label={[label distance=0.3mm]below:{\Large \( b \)}}] {};
    \node at (B1) [circle, fill, inner sep=1pt, label={[label distance=0mm]right:{\Large \(  \)}}] {};
    \node at (B2) [circle, fill, inner sep=1pt, label={[label distance=-1mm]above left:{\Large \(  \)}}] {};
    \node at (B3) [circle, fill, inner sep=1pt, label={[label distance=0mm]left:{\Large \(  \)}}] {};
    \node at (B4) [circle, fill, inner sep=1pt, label={[label distance=-1mm] above right:{\Large \(  \)}}] {};
    \node at (B) [circle, fill, inner sep=1pt, label={above:{\Large \(  \)}}] {};
    \node at (P) [circle, fill, inner sep=1pt, label={below:{\Large \(  \)}}] {};
    \node at (P0) [circle, fill, inner sep=1pt, label={[label distance=-1mm]above right:{\Large \( a \)}}] {};
    \node at (P1) [circle, fill, inner sep=1pt, label={[label distance=-1mm]below:{\Large \(  \)}}] {};
    \node at (P2) [circle, fill, inner sep=1pt, label={[label distance=-1mm]above left:{\Large \(  \)}}] {};
    \node at (P3) [circle, fill, inner sep=1pt, label={[label distance=-1mm]above left:{\Large \(  \)}}] {};
    \node at (P4) [circle, fill, inner sep=1pt, label={[label distance=-1mm]below left:{\Large \(  \)}}] {};
\end{tikzpicture}}
\captionof{figure}{\\ Planes that are far}
\label{planesicoopposite}

&

\resizebox{2.2cm}{!}{
\begin{tikzpicture}[baseline={(0, 4.2)}, scale=0.5]
    \def\pentagonradius{3} 
    \def\angle{72} 
    \def\rotationP{-90} 
    \def\rotationB{-54} 
    \def\verticalshift{8} 

    \begin{scope}[scale=1, yscale=0.6]
        \foreach \i in {0,1,2,3,4} {
            \coordinate (P\i) at (\i*\angle+\rotationP:\pentagonradius);
        }
        \draw[thick] (P0) -- (P1) -- (P2) -- (P3) -- (P4) -- cycle;

        \draw[thick] (P0) -- (P1);
        \draw[thick] (P1) -- (P2);
        \draw[thick] (P2) -- (P3);
        \draw[thick] (P3) -- (P4);
        \draw[thick] (P4) -- (P0);

        \coordinate (P) at (0,-5);
    \end{scope}

    \begin{scope}[scale=1, yscale=0.6, shift={(0,\verticalshift)}]
        \foreach \i in {0,1,2,3,4} {
            \coordinate (B\i) at (\i*\angle+\rotationB:\pentagonradius);
        }
        \draw[thick] (B0) -- (B1) -- (B2) -- (B3) -- (B4) -- cycle;

        \draw[thick] (B0) -- (B1);
        \draw[thick] (B1) -- (B2);
        \draw[thick] (B2) -- (B3);
        \draw[thick] (B3) -- (B4);
        \draw[thick] (B4) -- (B0);

        \coordinate (B) at (0,5);
    \end{scope}

    \draw[thick, opacity=0.3] (P0) -- (B0) -- (P1) -- (B1) -- (P2) -- (B2) -- (P3) -- (B3) -- (P4) -- (B4) -- (P0);

    \draw[thick, opacity=0.3] (P0) -- (P); 
    \draw[thick, opacity=0.3] (P1) -- (P); 
    \draw[thick, opacity=0.3] (P2) -- (P); 
    \draw[thick, opacity=0.3] (P3) -- (P); 
    \draw[thick, opacity=0.3] (P4) -- (P);

    \draw[thick, opacity=0.3] (B0) -- (B); 
    \draw[thick, opacity=0.3] (B1) -- (B); 
    \draw[thick, opacity=0.3] (B2) -- (B); 
    \draw[thick, opacity=0.3] (B3) -- (B); 
    \draw[thick, opacity=0.3] (B4) -- (B);

    \draw[color=black, line width=1pt] (B) -- (B0) -- (B4) -- cycle;
    \draw[color=black, line width=1pt] (P) -- (P4) -- (P0) -- cycle;
    \fill[color=black, opacity=0.3] (B) -- (B0) -- (B4) -- cycle;
    \fill[color=black, opacity=0.3] (P) -- (P0) -- (P4) -- cycle;

    \node at (B0) [circle, fill, inner sep=1pt, label={[label distance=0mm]above left:{\Large \(  \)}}] {};
    \node at (B1) [circle, fill, inner sep=1pt, label={[label distance=0mm]right:{\Large \(  \)}}] {};
    \node at (B2) [circle, fill, inner sep=1pt, label={[label distance=-1mm]above left:{\Large \(  \)}}] {};
    \node at (B3) [circle, fill, inner sep=1pt, label={[label distance=0mm]left:{\Large \(  \)}}] {};
    \node at (B4) [circle, fill, inner sep=1pt, label={[label distance=0.3mm]below:{\Large \( b \)}}] {};    
    \node at (B) [circle, fill, inner sep=1pt, label={above:{\Large \(  \)}}] {};
    \node at (P) [circle, fill, inner sep=1pt, label={below:{\Large \(  \)}}] {};
    \node at (P0) [circle, fill, inner sep=1pt, label={[label distance=-1mm]above right:{\Large \( a \)}}] {};
    \node at (P1) [circle, fill, inner sep=1pt, label={[label distance=-1mm]below:{\Large \(  \)}}] {};
    \node at (P2) [circle, fill, inner sep=1pt, label={[label distance=-1mm]above left:{\Large \(  \)}}] {};
    \node at (P3) [circle, fill, inner sep=1pt, label={[label distance=-1mm]above left:{\Large \(  \)}}] {};
    \node at (P4) [circle, fill, inner sep=1pt, label={[label distance=-1mm]below left:{\Large \(  \)}}] {};        
\end{tikzpicture}}
\captionof{figure}{\\ Planes that are close}\label{planesicoopposite}
\end{tabular}
\end{center}

\begin{lem}
If two planes $\alpha$ and $\beta$ in $\Delta$ do not intersect, then they are either opposite, far, or close.
\end{lem}

\begin{proof}
Since planes are maximal singular subspaces in $\Delta$, no point of $\alpha \cup \beta$ is collinear to all of $\alpha \cup \beta$.
Thus, if $a$ is a point of $\alpha$, then $a^{\perp} \cap \beta$ can at most be a line.

\textbf{Close planes:} Suppose there exists a point $a \in \alpha$ such that $a^{\perp} \cap \beta$ is a line $L$.
Then every point of $\alpha \setminus \{a\}$ has distance at most $2$ to all points of $L$.
Let $c$ be a point in $\beta \setminus L$. Then $a$ and $c$ have distance $2$ and $a^\perp \cap c^\perp = L$.
The point $c$ cannot be collinear to any point $x$ of $\alpha$, 
since otherwise $x$ would have to be contained in $L$, forcing $\alpha$ and $\beta$ to intersect.
The point $c$ can also not be opposite to all points of $\alpha \setminus \{a\}$, since otherwise there would exist 
a line $J \subset \alpha$ not containing $a$ such that all points of $J$ have distance $3$ to $c$.
Thus, $c$ has distance $2$ to at least a line $K$ of $\alpha$ through $a$.

Suppose $c$ had distance $2$ to all points of $\alpha \setminus \{a\}$.

Let $x \in \alpha$ be a point distinct from $a$, and let $\ell$ be a point on $L$.
Then $x^\perp \cap c^\perp$ and $x^\perp \cap \ell^\perp$ are lines, intersecting in a unique point $z$ such that
$x \perp z \perp c\ell$.
Since $x \perp a \perp \ell$, the line $x^\perp \cap \ell^\perp$ contains $a$ and $z$, and thus $a \perp z$.
Then $z \in a^\perp \cap c^\perp$, and thus $z \in L$.

Now let $y$ be a point in $\alpha$ not on the line $ax$.
Then we can use the same arguments as before to see that $y$ must be collinear to a point $z'$ of $L$.
In $\Res_{\Delta} (a)$ we can see the lines $yx$, $xz$, $yz'$, and $zz'$, and they would form a quadrangle,
unless $x \perp z'$ or $y \perp z$ or $z = z'$ -- but in these cases $\Res_{\Delta} (a)$ would contain triangles, which is even worse.

It follows that $c$ has distance $2$ to all points of a unique line $K$ of $\alpha$ through $a$, and distance $3$ to all other points of $\alpha$.
Let $b$ be the unique point such that $c \perp b \perp K$.
Since $a^\perp \cap c^\perp = L$, $b$ has to be contained in $L$.

Now we have $d \perp b \perp K$, 
and it follows that $d^{\pperp} \cap \alpha = L$, for all $d \in \beta \setminus L$.
One shows analogously that every point $e$ of $\alpha \setminus K$ has distance $2$ to $K$,
and distance $3$ to all other points of $\beta$.

Suppose another point $e \in K$ were collinear to a point $l$ of $L$.
Then $e$, $a$, and $l$ would form a triangle in $\Res_{\Delta}(b)$,
contradicting the fact that all point residues are generalized pentagons.
Analogously, no point of $K$ except $b$ is collinear to a point of $L$.
Thus $\alpha$ and $\beta$ are close.

\textbf{Far planes:} Suppose no point of $\alpha$ is collinear to a line of $\beta$, 
but there exists a point $a \in \alpha$ such that $a^{\perp} \cap \beta$ is a point $b$.

Suppose there existed another point $a'$ collinear to a point $b'$ of $\beta$.
If $b = b'$, then $b \perp aa'$ and we are in the case of close planes with the roles of $\alpha$ and $\beta$ switched.
So suppose $b \neq b'$. Then $b' \in a^{\pperp} \cap \beta$, and $b'^\perp \cap a^{\perp}$ contains $a'$ and $b$,
forcing $a'$ and $b$ to be collinear. But then $a' \perp bb'$, 
contradicting the fact that no point of $\alpha$ is collinear to a line of $\beta$.
It follows that all points of $\alpha \setminus \{a\}$ have distance at least $2$ to all points of $\beta$.

Let $x \in \alpha$ be a point distinct from $a$ and suppose $x$ had distance $2$ to all points of $\beta$.
Let $K$ be a line in $\beta$ not containing $b$. Then there exists a unique point $z$ such that $b \perp z \perp K$.
Let $c$ and $d$ be two distinct points on $K$. 
Then there also has to exist a unique point $z'$ collinear to $x$ and $bc$,
and a unique point $z''$ collinear to $x$ and $bd$.
Then
$$x^\perp \cap c^\perp = zz', \quad x^\perp \cap b^\perp = z'z'', \quad x^\perp \cap d^\perp = zz''.$$
If $z$, $z'$, and $z''$ were all distinct, they would form a triangle in $\Res_{\Delta}(x)$.
If at least two of them coincided, say $z = z'$, then $z$ would be collinear to the entire plane $\beta$, 
contradicting the fact that planes are maximal singular subspaces.

\textbf{Opposite planes:} Suppose no point of $\alpha$ is collinear to any point of $\beta$.
No point $a$ of $\alpha \cup \beta$ can be opposite all points of either $\alpha$ or $\beta$, or have distance $2$
to only a unique point of either $\alpha$ or $\beta$, since otherwise there would exist a line all of whose points at distance $3$ to $a$.
\end{proof}

\begin{coro}
Let $\alpha$ be a plane, and $p$ a point in $\Delta$. Then $p$ cannot have distance $2$ to all points of $\alpha$.
\end{coro}

\begin{defn}
A \emph{chamber} in $\Delta$ is a triple $(a,L,\alpha)$, consisting of a point $a$, a line $L$, and a plane $\alpha$ such that $a \in L \subseteq \alpha$.
\end{defn}

\begin{prop}
Any two chambers of $\Delta$ are contained in a common apartment.
\end{prop}

\begin{proof}
Let $\alpha$ and $\beta$ be two planes. 
Let $L$ be a line in $\alpha$, and $a$ a point on $L$. 
Let $K$ be a line in $\beta$, and $b$ a point on $K$.

\textbf{Planes intersecting in a line:} %
Suppose $\alpha$ and $\beta$ intersect in a line $M$.
If $K=L=M$, we can find a point $p$ opposite $a$ 
and an apartment $\Sigma$ through $a$ and $p$ containing $K=L$, $\alpha$, and $\beta$ (according to Lemma \ref{canfindapartmentthroughoppopoints}).
If $b$ is not contained in $\Sigma$, we can apply Lemma \ref{findpointoppoanddist2toline}
to find another point $p'$ opposite $a$ and at distance $2$ to $b$ and an apartment 
$\Sigma'$ through $a$ and $p'$ containing $b$, $L=K$, $\alpha$ and $\beta$ (again, according to Lemma \ref{canfindapartmentthroughoppopoints}).

In the following we use Lemma \ref{canfindapartmentthroughoppopoints} without mention.

Suppose $L \neq M$.
Let $m$ be a point on $M$ and $L$.
Then there exists a point $n$ opposite $m$ having distance $2$ to all points of a unique line of $\alpha$, 
and to all points of a unique line of $\beta$.
Now we can find an apartment $\Sigma_0$ 
containing $m$, $n$, $n^{\pperp} \cap \alpha$, and $n^{\pperp} \cap \beta$, 
and thus containing $\alpha$ and $\beta$ as well.

Suppose $L$ and $K$ are not contained in $\Sigma_0$.

For convenience, we now label all points of $\Sigma_0$. 
Let $p_0$, $p_1$, $p_2$, $p_3$, and $p_4$ be the points of $m^{\perp} \cap n^{\pperp}$ such that 
$p_4$ and $p_0$ are in $\alpha$, $p_0$ and $p_1$ are in $\beta$, and $p_4 \perp p_3 \perp p_2 \perp p_1$.
Let $b_0$, $b_1$, $b_2$, $b_3$, and $b_4$ be the points of $n^{\perp} \cap m^{\pperp}$ such that 
$p_i \perp b_{i-1}b_i$ for $i \in \Z/5\Z$ (cf.~Figure \ref{1containedinapartmentschritt1}).

Let $p_4'$ be the intersection point of $L$ and $p_0p_4$.
Then $p_4'$ is opposite $b_1$, and we have
$p_0, m \in p_4'^{\perp} \cap b_1^{\pperp}$, $p_1 \in p_4'^{\pperp} \cap b_1^{\perp}$.
Thus, we can find an apartment $\Sigma$ through $p_4'$ and $b_1$ containing $L$, $\alpha$, and $\beta$.

\begin{center}
\resizebox{2.7cm}{!}{
\begin{tikzpicture}[baseline={(0, 4.2)}, scale=0.5]

    \def\pentagonradius{3} 
    \def\angle{72} 
    \def\rotationP{-90} 
    \def\rotationB{-54} 
    \def\verticalshift{8} 

    \begin{scope}[scale=1, yscale=0.6]
        \foreach \i in {0,1,2,3,4} {
            \coordinate (P\i) at (\i*\angle+\rotationP:\pentagonradius);
        }
        \draw[thick, opacity=0.3] (P0) -- (P1) -- (P2) -- (P3) -- (P4) -- cycle;

        \coordinate (P) at (0,-5);
    \end{scope}

    \begin{scope}[scale=1, yscale=0.6, shift={(0,\verticalshift)}]
        \foreach \i in {0,1,2,3,4} {
            \coordinate (B\i) at (\i*\angle+\rotationB:\pentagonradius);
        }
        \draw[thick, opacity=0.3] (B0) -- (B1) -- (B2) -- (B3) -- (B4) -- cycle;

        \coordinate (B) at (0,5);
    \end{scope}

    \draw[thick, opacity=0.3] (P0) -- (B0) -- (P1) -- (B1) -- (P2) -- (B2) 
                              -- (P3) -- (B3) -- (P4) -- (B4) -- (P0);

    \draw[thick, opacity=0.3] (P0) -- (P); 
    \draw[thick, opacity=0.3] (P1) -- (P); 
    \draw[thick, opacity=0.3] (P2) -- (P); 
    \draw[thick, opacity=0.3] (P3) -- (P); 
    \draw[thick, opacity=0.3] (P4) -- (P);

    \draw[thick, opacity=0.3] (B0) -- (B); 
    \draw[thick, opacity=0.3] (B1) -- (B); 
    \draw[thick, opacity=0.3] (B2) -- (B); 
    \draw[thick, opacity=0.3] (B3) -- (B); 
    \draw[thick, opacity=0.3] (B4) -- (B);

    \draw[color=black, line width=1pt] (P0) -- (P) -- (P1) -- cycle;
    \draw[color=black, line width=1pt] (P) -- (P4) -- (P0) -- cycle;
    \fill[color=black, opacity=0.3] (P0) -- (P) -- (P1) -- cycle;
    \fill[color=black, opacity=0.3] (P) -- (P0) -- (P4) -- cycle;

    \coordinate (beta) at ($ (P) + 0.5*(P1) - 0.5*(P) $); 
    \coordinate (alpha) at ($ (P) + 0.5*(P4) - 0.5*(P) $); 

    \node at (B0) [circle, fill, inner sep=1pt, label={[label distance=0mm]above left:{\Large \( b_0 \)}}] {};
    \node at (B1) [circle, fill, inner sep=1pt, label={[label distance=0mm]right:{\Large \( b_1 \)}}] {};
    \node at (B2) [circle, fill, inner sep=1pt, label={[label distance=0.5mm]below:{\Large \( b_2 \)}}] {};
    \node at (B3) [circle, fill, inner sep=1pt, label={[label distance=0mm]left:{\Large \( b_3 \)}}] {};
    \node at (B4) [circle, fill, inner sep=1pt, label={[label distance=0mm]above right:{\Large \( b_4 \)}}] {};    

    \node at (B)  [circle, fill, inner sep=1pt, label={above:{\Large \( n \)}}] {};
    \node at (P)  [circle, fill, inner sep=1pt, label={below:{\Large \( m \)}}] {};

    \node at (P0) [circle, fill, inner sep=1pt, label={[label distance=0mm]above:{\Large \( p_0 \)}}] {};
    \node at (P1) [circle, fill, inner sep=1pt, label={[label distance=0mm]right:{\Large \( p_1 \)}}] {};
    \node at (P2) [circle, fill, inner sep=1pt, label={[label distance=-1mm]above left:{\Large \( p_2 \)}}] {};
    \node at (P3) [circle, fill, inner sep=1pt, label={[label distance=-1mm]above right:{\Large \( p_3 \)}}] {};
    \node at (P4) [circle, fill, inner sep=1pt, label={[label distance=0mm]left:{\Large \( p_4 \)}}] {};        

    \node at (alpha) [label={[label distance=-1mm]below left:{\Large \( \alpha \)}}] {};  
    \node at (beta)  [label={[label distance=-1mm]below right:{\Large \( \beta \)}}] {};   


\end{tikzpicture}}
\captionof{figure}{}\label{1containedinapartmentschritt1}
\end{center}

Now suppose $K$ is not contained in $\Sigma$. 
By abuse of notation, we label all the points of $\Sigma$ as we did before for $\Sigma_0$ (cf.~Figure \ref{1containedinapartmentschritt1}).
Then $L$ is the line $mp_4$. 

Put $p_1' := K \cap p_0p_1$.
Then $p_1'$ and $b_3$ are opposite, and we can find an apartment $\Sigma'$ through $p_1'$ and $b_3$
containing $\alpha$, $\beta$, $L$, and $K$.

Again, by abuse of notation, we label all the points of $\Sigma'$ as we did before for $\Sigma_0$ (cf.~Figure \ref{1containedinapartmentschritt1}).
Then $L$ is either the line $p_0p_4$ or the line $mp_4$, and $K$ is either the line $mp_1$ or $p_0p_1$.

We still have to show that we can find apartments containing $L$, $K$, $\alpha$, and $\beta$ that also contain the points $a$ and $b$.

If $K = p_0p_1$ and $b \neq p_0, p_1$, then we can do exactly the same as in (ii') and find an apartment that contains $b$.
If $L = p_0p_4$ and $a \neq p_0, p_4$, then — because of symmetry — we can also do exactly the same as in (ii') with roles switched.

Suppose $K = mp_1$ and $b \neq m, p_1$. Then $b$ is opposite $b_3$, $b^\perp \cap b_3^{\pperp}$ contains $p_4$, and $b_3^\perp \cap b^{\pperp}$ contains $p_0$ and $m$, and thus we can find an apartment through $b$ and $b_3$ containing $b$, $L$, $K$, $\alpha$, and $\beta$.
If $L = mp_4$ and $a \neq m, p_4$, we do exactly the same.

\textbf{Planes intersecting in a point:} %
Suppose $\alpha$ and $\beta$ intersect in a point $i$.
Let $j$ be a point opposite $i$. Then $j$ has distance $2$ to a unique line of $\alpha$ and a unique line of $\beta$.
As before, we can find an apartment $\Sigma$ containing $i$, $j$, $j^{\pperp} \cap \alpha$, $j^{\pperp} \cap \beta$, and thus $\alpha$ and $\beta$.

Suppose $L$ and $K$ are not contained in $\Sigma$.

(i) Suppose $L$ is disjoint from $i$. 
Then we can use Lemma \ref{findpointoppoanddist2toline} to change $j$ to a point $j'$ opposite $i$ such that $j'^{\pperp} \cap \alpha = L$.
Then we can find an apartment $\Sigma'$ containing $i$, $j'$, $L$, $j'^{\pperp} \cap \beta$, $\alpha$, and $\beta$.

(ii) Suppose $L$ is a line through $i$. 
Then we can find an apartment $\Sigma'$ through $i$ and $j$ containing $L$, $\alpha$ and $\beta$.

So in both cases (i) and (ii) we found an apartment $\Sigma'$ containing $\alpha$, $\beta$, and $L$.
For convenience, we now label all points of $\Sigma'$. 
Let $b_0$, $b_1$, $b_2$, $b_3$, and $b_4$ be the points of $j'^{\pperp} \cap i^{\perp}$ such that 
$b_0$ and $b_1$ are in $\beta$, $b_3$ and $b_4$ are in $\alpha$, $b_0 \perp b_4$, and $b_1 \perp b_2 \perp b_3$.
Let $p_0$, $p_1$, $p_2$, $p_3$, and $p_4$ be the points of $j'^{\perp} \cap i^{\pperp}$ such that 
$p_i \perp b_{i-1}b_i$ for $i \in \Z/5\Z$ (cf.~Figure \ref{containedinapartmentschritt2case1pic1}).
Now we know that the line $L$ is either equal to $b_3b_4$ or to $ib_3$, according to (i) and (ii).

Suppose $\Sigma'$ does not contain $K$. Then we have to consider two cases again.

(i') Suppose $K$ is disjoint from $i$. 

Let $b_0'$ be the intersection point of $K$ with $ib_0$, and let $b_1'$ be the intersection point of $K$ with $ib_1$ (cf.~Figure \ref{containedinapartmentschritt2case1pic2}).

Note that $b_1'$ and $p_4$ are opposite. 
Hence, $b_1'^{\perp} \cap p_4^{\pperp}$ is a generalized pentagon, and $b_1'^{\pperp} \cap p_4^{\perp}$ is a generalized pentagon.
The line $b_3b_4$ is contained in $b_1'^{\pperp} \cap p_4^{\perp}$.
The line $ib_0'$ is contained in $b_1'^{\perp} \cap p_4^{\pperp}$.
Hence, we can find an apartment $\Sigma''$ through $p_4$ and $b_1'$ containing $i$, $b_0'$, $b_3$, $b_4$, $\alpha$, and $\beta$ (cf.~Figure \ref{containedinapartmentschritt2case1pic3}).
In particular, since $b_0'$ and $b_1'$ are contained in $\Sigma''$, $K$ is contained in $\Sigma''$, 
and since $ib_3$ and $b_3b_4$ are contained in $\Sigma''$, $L$ is contained in $\Sigma''$.

\begin{center}
\begin{tabular}{>{\centering\arraybackslash} m{3.1cm} >{\centering\arraybackslash} m{3.1cm} >{\centering\arraybackslash} m{3.1cm}}

\resizebox{2.7cm}{!}{
\begin{tikzpicture}[baseline={(0, 4.2)}, scale=0.5]

    \def\pentagonradius{3} 
    \def\angle{72} 
    \def\rotationP{-90} 
    \def\rotationB{-54} 
    \def\verticalshift{8} 

    \begin{scope}[scale=1, yscale=0.6]
        \foreach \i in {0,1,2,3,4} {
            \coordinate (P\i) at (\i*\angle+\rotationP:\pentagonradius);
        }
        \draw[thick] (P0) -- (P1) -- (P2) -- (P3) -- (P4) -- cycle;
        \coordinate (P) at (0,-5);
    \end{scope}

    \begin{scope}[scale=1, yscale=0.6, shift={(0,\verticalshift)}]
        \foreach \i in {0,1,2,3,4} {
            \coordinate (B\i) at (\i*\angle+\rotationB:\pentagonradius);
        }
        \draw[thick] (B0) -- (B1) -- (B2) -- (B3) -- (B4) -- cycle;
        \coordinate (B) at (0,5);
    \end{scope}

    \draw[thick, opacity=0.3]
        (P0) -- (B0) -- (P1) -- (B1) -- (P2) -- (B2)
        -- (P3) -- (B3) -- (P4) -- (B4) -- (P0);

    \foreach \i in {0,1,2,3,4} {
        \draw[thick, opacity=0.3] (P\i) -- (P);
    }

    \foreach \i in {0,1,2,3,4} {
        \draw[thick, opacity=0.3] (B\i) -- (B);
    }

    \draw[color=black, line width=1pt] (B) -- (B0) -- (B1) -- cycle;
    \draw[color=black, line width=1pt] (B) -- (B3) -- (B4) -- cycle;
    \fill[color=black, opacity=0.3] (B) -- (B3) -- (B4) -- cycle;
    \fill[color=black, opacity=0.3] (B) -- (B0) -- (B1) -- cycle;

    \coordinate (B0') at ($ (B) + 0.8*(B0) - 0.8*(B) $); 
    \coordinate (B1') at ($ (B) + 0.8*(B1) - 0.8*(B) $); 
    \coordinate (K)   at ($ (B0) + 0.5*(B1) - 0.5*(B0) $); 
    \coordinate (beta)  at ($ (B) + 0.5*(B1) - 0.5*(B) $); 
    \coordinate (alpha) at ($ (B) + 0.5*(B3) - 0.5*(B) $); 

    \node at (B0) [circle, fill, inner sep=1pt, label={[label distance=-1mm, xshift=-0.5mm]above left:{\Large \( b_0 \)}}] {};
    \node at (B1) [circle, fill, inner sep=1pt, label={[label distance=-1mm]right:{\Large \( b_1 \)}}] {};
    \node at (B2) [circle, fill, inner sep=1pt, label={[label distance=0.5mm]below:{\Large \( b_2 \)}}] {};
    \node at (B3) [circle, fill, inner sep=1pt, label={[label distance=-1mm]left:{\Large \( b_3 \)}}] {};
    \node at (B4) [circle, fill, inner sep=1pt, label={[label distance=-1mm, xshift=1mm]above right:{\Large \( b_4 \)}}] {};
    \node at (B)  [circle, fill, inner sep=1pt, label={above:{\Large \( i \)}}] {};
    \node at (P)  [circle, fill, inner sep=1pt, label={below:{\Large \( j' \)}}] {};

    \node at (P0) [circle, fill, inner sep=1pt, label={[label distance=0mm]above:{\Large \( p_0 \)}}] {};
    \node at (P1) [circle, fill, inner sep=1pt, label={[label distance=-1mm]right:{\Large \( p_1 \)}}] {};
    \node at (P2) [circle, fill, inner sep=1pt, label={[label distance=-1mm]above left:{\Large \( p_2 \)}}] {};
    \node at (P3) [circle, fill, inner sep=1pt, label={[label distance=-1mm]above right:{\Large \( p_3 \)}}] {};
    \node at (P4) [circle, fill, inner sep=1pt, label={[label distance=-1mm]left:{\Large \( p_4 \)}}] {};        
    \node at (alpha) [label={[label distance=-1mm]above left:{\Large \( \alpha \)}}] {};  
    \node at (beta)  [label={[label distance=-1mm]above right:{\Large \( \beta \)}}] {};  

\end{tikzpicture}}
\captionof{figure}{\\ $\Sigma'$ with all points labeled}
\label{containedinapartmentschritt2case1pic1}

& 

\resizebox{2.7cm}{!}{
\begin{tikzpicture}[baseline={(0, 4.2)}, scale=0.5]

    \def\pentagonradius{3} 
    \def\angle{72} 
    \def\rotationP{-90} 
    \def\rotationB{-54} 
    \def\verticalshift{8} 

    \begin{scope}[scale=1, yscale=0.6]
        \foreach \i in {0,1,2,3,4} {
            \coordinate (P\i) at (\i*\angle+\rotationP:\pentagonradius);
        }
        \draw[thick, opacity=0.3] (P0) -- (P1) -- (P2) -- (P3) -- (P4) -- cycle;
        \coordinate (P) at (0,-5);
    \end{scope}

    \begin{scope}[scale=1, yscale=0.6, shift={(0,\verticalshift)}]
        \foreach \i in {0,1,2,3,4} {
            \coordinate (B\i) at (\i*\angle+\rotationB:\pentagonradius);
        }
        \draw[thick, opacity=0.3] (B0) -- (B1) -- (B2) -- (B3) -- (B4) -- cycle;
        \coordinate (B) at (0,5);
    \end{scope}

    \draw[thick, opacity=0.3]
        (P0) -- (B0) -- (P1) -- (B1) -- (P2) -- (B2)
        -- (P3) -- (B3) -- (P4) -- (B4) -- (P0);

    \foreach \i in {0,1,2,3,4} {
        \draw[thick, opacity=0.3] (P\i) -- (P);
    }

    \foreach \i in {0,1,2,3,4} {
        \draw[thick, opacity=0.3] (B\i) -- (B);
    }

    \draw[color=black, line width=1pt] (B) -- (B0) -- (B1) -- cycle;
    \draw[color=black, line width=1pt] (B) -- (B3) -- (B4) -- cycle;
    \fill[color=black, opacity=0.3] (B) -- (B3) -- (B4) -- cycle;
    \fill[color=black, opacity=0.3] (B) -- (B0) -- (B1) -- cycle;

    \coordinate (B0') at ($ (B) + 0.8*(B0) - 0.8*(B) $); 
    \coordinate (B1') at ($ (B) + 0.8*(B1) - 0.8*(B) $); 
    \coordinate (K)   at ($ (B0) + 0.5*(B1) - 0.5*(B0) $); 
    \draw[thick, opacity=0.3] (B4) -- (B0');
    \draw[thick, opacity=0.3] (B2) -- (B1');

    \coordinate (beta)  at ($ (B) + 0.5*(B1) - 0.5*(B) $); 
    \coordinate (alpha) at ($ (B) + 0.5*(B3) - 0.5*(B) $); 

    \node at (B0) [circle, fill, inner sep=1pt, label={[label distance=-1mm, xshift=-0.5mm]above left:{\Large \( b_0 \)}}] {};
    \node at (B1) [circle, fill, inner sep=1pt, label={[label distance=-1mm]right:{\Large \( b_1 \)}}] {};
    \node at (B2) [circle, fill, inner sep=1pt, label={[label distance=0.5mm]below:{\Large \( b_2 \)}}] {};
    \node at (B3) [circle, fill, inner sep=1pt, label={[label distance=-1mm]left:{\Large \( b_3 \)}}] {};
    \node at (B4) [circle, fill, inner sep=1pt, label={[label distance=-1mm, xshift=1mm]above right:{\Large \( b_4 \)}}] {};
    \node at (B)  [circle, fill, inner sep=1pt, label={above:{\Large \( i \)}}] {};
    \node at (P)  [circle, fill, inner sep=1pt, label={below:{\Large \( j' \)}}] {};
    \node at (P0) [circle, fill, inner sep=1pt, label={[label distance=0mm]above:{\Large \( p_0 \)}}] {};
    \node at (P1) [circle, fill, inner sep=1pt, label={[label distance=-1mm]right:{\Large \( p_1 \)}}] {};
    \node at (P2) [circle, fill, inner sep=1pt, label={[label distance=-1mm]above left:{\Large \( p_2 \)}}] {};
    \node at (P3) [circle, fill, inner sep=1pt, label={[label distance=-1mm]above right:{\Large \( p_3 \)}}] {};
    \node at (P4) [circle, fill, inner sep=1pt, label={[label distance=-1mm]left:{\Large \( p_4 \)}}] {};        
    \node at (alpha) [label={[label distance=-1mm]above left:{\Large \( \alpha \)}}] {};  
    \node at (beta)  [label={[label distance=-1mm]above right:{\Large \( \beta \)}}] {};  

    \node at (B0') [circle, fill, inner sep=1pt, label={[label distance=0mm, yshift=-1mm]above left:{\Large \( b_0' \)}}] {};
    \node at (B1') [circle, fill, inner sep=1pt, label={[label distance=0mm]right:{\Large \( b_1' \)}}] {};
    \node at (K)   [circle, fill, inner sep=1pt, label={[label distance=-1mm]above left:{\Large \( k \)}}] {};      

\end{tikzpicture}}
\captionof{figure}{\\ $\Sigma'$ in $\Delta$ in case (i')}
\label{containedinapartmentschritt2case1pic2}

& 

\resizebox{2.7cm}{!}{
\begin{tikzpicture}[baseline={(0, 4.2)}, scale=0.5]
    \def\pentagonradius{3}
    \def\angle{72}
    \def\rotationP{-90}
    \def\rotationB{-54}
    \def\verticalshift{8}

    \begin{scope}[scale=1, yscale=0.6]
        \foreach \i in {0,1,2,3,4} {
            \coordinate (P\i) at (\i*\angle+\rotationP:\pentagonradius);
        }
        \draw[thick, opacity=0.3] (P0) -- (P1) -- (P2) -- (P3) -- (P4) -- cycle;

        \coordinate (P) at (0,-5);
    \end{scope}

    \begin{scope}[scale=1, yscale=0.6, shift={(0,\verticalshift)}]
        \foreach \i in {0,1,2,3,4} {
            \coordinate (B\i) at (\i*\angle+\rotationB:\pentagonradius);
        }
        \draw[thick, opacity=0.3] (B0) -- (B1) -- (B2) -- (B3) -- (B4) -- cycle;

        \coordinate (B) at (0,5);
    \end{scope}

    \draw[thick, opacity=0.3] (P0) -- (B0) -- (P1) -- (B1) -- (P2) -- (B2)
                               -- (P3) -- (B3) -- (P4) -- (B4) -- (P0);

    \draw[thick, opacity=0.3] (P0) -- (P);
    \draw[thick, opacity=0.3] (P1) -- (P);
    \draw[thick, opacity=0.3] (P2) -- (P);
    \draw[thick, opacity=0.3] (P3) -- (P);
    \draw[thick, opacity=0.3] (P4) -- (P);

    \draw[thick, opacity=0.3] (B0) -- (B);
    \draw[thick, opacity=0.3] (B1) -- (B);
    \draw[thick, opacity=0.3] (B2) -- (B);
    \draw[thick, opacity=0.3] (B3) -- (B);
    \draw[thick, opacity=0.3] (B4) -- (B);

    \draw[color=black, line width=1pt] (B) -- (B0) -- (B1) -- cycle;
    \draw[color=black, line width=1pt] (B) -- (B3) -- (B4) -- cycle;
    \fill[color=black, opacity=0.3] (B) -- (B3) -- (B4) -- cycle;
    \fill[color=black, opacity=0.3] (B) -- (B0) -- (B1) -- cycle;

    \coordinate (B0') at ($ (B) + 0.8*(B0) - 0.8*(B) $);
    \coordinate (B1') at ($ (B) + 0.8*(B1) - 0.8*(B) $);
    \coordinate (K)   at ($ (B0) + 0.5*(B1) - 0.5*(B0) $);

    \coordinate (beta)  at ($ (B) + 0.5*(B1) - 0.5*(B) $);
    \coordinate (alpha) at ($ (B) + 0.5*(B3) - 0.5*(B) $);


    \node at (B0) [circle, fill, inner sep=1pt,
        label={[label distance=-1mm, xshift=-0.5mm]above left:{\Large \( b_0' \)}}] {};
    \node at (B1) [circle, fill, inner sep=1pt,
        label={[label distance=0mm]right:{\Large \( b_1' \)}}] {};
    \node at (B2) [circle, fill, inner sep=1pt,
        label={[label distance=0.5mm]below:{\Large \( b_2 \)}}] {};
    \node at (B3) [circle, fill, inner sep=1pt,
        label={[label distance=0mm]left:{\Large \( b_3 \)}}] {};
    \node at (B4) [circle, fill, inner sep=1pt,
        label={[label distance=-1mm, xshift=1mm]above right:{\Large \( b_4 \)}}] {};
    \node at (B) [circle, fill, inner sep=1pt,
        label={above:{\Large \( i \)}}] {};
    \node at (P) [circle, fill, inner sep=1pt,
        label={below:{\Large \(  \)}}] {};
    \node at (P0) [circle, fill, inner sep=1pt,
        label={[label distance=0mm]above:{\Large \(  \)}}] {};
    \node at (P1) [circle, fill, inner sep=1pt,
        label={[label distance=-1mm]right:{\Large \(  \)}}] {};
    \node at (P2) [circle, fill, inner sep=1pt,
        label={[label distance=-1mm]above left:{\Large \(  \)}}] {};
    \node at (P3) [circle, fill, inner sep=1pt,
        label={[label distance=-1mm]above right:{\Large \( p_3 \)}}] {};
    \node at (P4) [circle, fill, inner sep=1pt,
        label={[label distance=-1mm]left:{\Large \( p_4 \)}}] {};
    \node at (alpha) [label={[label distance=-1mm]above left:{\Large \( \alpha \)}}] {};
    \node at (beta)  [label={[label distance=-1mm]above right:{\Large \( \beta \)}}] {};

\end{tikzpicture}}
\captionof{figure}{\\ Apartment $\Sigma''$ of (i')}\label{containedinapartmentschritt2case1pic3}

\end{tabular}
\end{center}

(ii') Suppose $K$ is a line through $i$.

Let $b_1'$ be the intersection point $K \cap b_0b_1$ (cf.~Figure \ref{containedinapartmentschritt2case2pic1}).
Then $b_1'$ is opposite both $p_3$ and $p_4$.
Both $b_1'^{\perp} \cap p_4^{\pperp}$ and $b_1'^{\pperp} \cap p_4^{\perp}$ are generalized quadrangles.
The points $b_3$ and $b_4$ are contained in $b_1'^{\pperp} \cap p_4^{\perp}$, and $b_0$ and $i$ are contained in $b_1'^{\perp} \cap p_4^{\pperp}$.
Hence, we can find an apartment $\Sigma''$ through $p_4$ and $b_1'$ containing $b_3$, $b_4$, $i$, and $b_0$.
Then $\Sigma''$ contains $L$, $K$, $\alpha$, and $\beta$.

\begin{center}
    \resizebox{2.7cm}{!}{
        \begin{tikzpicture}[baseline={(0, 4.2)}, scale=0.5]
            \def\pentagonradius{3} 
            \def\angle{72} 
            \def\rotationP{-90} 
            \def\rotationB{-54} 
            \def\verticalshift{8} 

            \begin{scope}[scale=1, yscale=0.6]
                \foreach \i in {0,1,2,3,4} {
                    \coordinate (P\i) at (\i*\angle+\rotationP:\pentagonradius);
                }
                \draw[thick, opacity=0.3] (P0) -- (P1) -- (P2) -- (P3) -- (P4) -- cycle;

                \coordinate (P) at (0,-5);
            \end{scope}

            \begin{scope}[scale=1, yscale=0.6, shift={(0,\verticalshift)}]
                \foreach \i in {0,1,2,3,4} {
                    \coordinate (B\i) at (\i*\angle+\rotationB:\pentagonradius);
                }
                \draw[thick, opacity=0.3] (B0) -- (B1) -- (B2) -- (B3) -- (B4) -- cycle;

                \coordinate (B) at (0,5);
            \end{scope}

            \draw[thick, opacity=0.3] (P0) -- (B0) -- (P1) -- (B1) -- (P2) -- (B2)
                                     -- (P3) -- (B3) -- (P4) -- (B4) -- (P0);

            \draw[thick, opacity=0.3] (P0) -- (P); 
            \draw[thick, opacity=0.3] (P1) -- (P); 
            \draw[thick, opacity=0.3] (P2) -- (P); 
            \draw[thick, opacity=0.3] (P3) -- (P); 
            \draw[thick, opacity=0.3] (P4) -- (P);

            \draw[thick, opacity=0.3] (B0) -- (B); 
            \draw[thick, opacity=0.3] (B1) -- (B); 
            \draw[thick, opacity=0.3] (B2) -- (B); 
            \draw[thick, opacity=0.3] (B3) -- (B); 
            \draw[thick, opacity=0.3] (B4) -- (B);

            \draw[color=black, line width=1pt] (B) -- (B0) -- (B1) -- cycle;
            \draw[color=black, line width=1pt] (B) -- (B3) -- (B4) -- cycle;
            \fill[color=black, opacity=0.3] (B) -- (B3) -- (B4) -- cycle;
            \fill[color=black, opacity=0.3] (B) -- (B0) -- (B1) -- cycle;

            \coordinate (B1') at ($ (B1) + 0.2*(B0) - 0.2*(B1) $); 
            \draw[thick, opacity=0.3] (B) -- (B1');

            \coordinate (beta) at ($ (B) + 0.5*(B1) - 0.5*(B) $); 
            \coordinate (alpha) at ($ (B) + 0.5*(B3) - 0.5*(B) $); 

            \node at (B0) [circle, fill, inner sep=1pt, label={[label distance=0mm]above left:{\Large \( b_0 \)}}] {};
            \node at (B1) [circle, fill, inner sep=1pt, label={[label distance=0mm]right:{\Large \( b_1 \)}}] {};
            \node at (B2) [circle, fill, inner sep=1pt, label={[label distance=0.5mm]below:{\Large \( b_2 \)}}] {};
            \node at (B3) [circle, fill, inner sep=1pt, label={[label distance=-1mm]left:{\Large \( b_3 \)}}] {};
            \node at (B4) [circle, fill, inner sep=1pt, label={[label distance=0.3mm]below:{\Large \( b_4 \)}}] {};    
            \node at (B)  [circle, fill, inner sep=1pt, label={above:{\Large \( i \)}}] {};
            \node at (P)  [circle, fill, inner sep=1pt, label={below:{\Large \( j' \)}}] {};
            \node at (P0) [circle, fill, inner sep=1pt, label={[label distance=0mm]above:{\Large \( p_0 \)}}] {};
            \node at (P1) [circle, fill, inner sep=1pt, label={[label distance=-1mm]right:{\Large \( p_1 \)}}] {};
            \node at (P2) [circle, fill, inner sep=1pt, label={[label distance=-1mm]above left:{\Large \( p_2 \)}}] {};
            \node at (P3) [circle, fill, inner sep=1pt, label={[label distance=-1mm]above left:{\Large \( p_3 \)}}] {};
            \node at (P4) [circle, fill, inner sep=1pt, label={[label distance=-1mm]below left:{\Large \( p_4 \)}}] {};        
            \node at (alpha) [label={[label distance=-1mm]above left:{\Large \( \alpha \)}}] {};  
            \node at (beta)  [label={[label distance=-1mm]above right:{\Large \( \beta \)}}] {};  
            \node at (B1') [circle, fill, inner sep=1pt, label={[label distance=0mm]below right:{\Large \( b_1' \)}}] {};      

        \end{tikzpicture}
    }
    \captionof{figure}{\\ $\Sigma'$ in $\Delta$ in case (ii')}\label{containedinapartmentschritt2case2pic1}
\end{center}

Now that we have found an apartment again containing $L$, $K$, $\alpha$, and $\beta$, 
it follows from similar arguments as before that we can also include $a$ and $b$.

\textbf{Non-intersecting planes:} %
Suppose $\alpha$ and $\beta$ are either close, far, or opposite. 
Let $a \in \alpha$ and $b \in \beta$ be two opposite points.
Then there exists an apartment through $a$ and $b$ containing $\alpha$ and $\beta$.

With similar arguments as above, one can adjust the apartment such that it contains $L$, $K$, $a$, and $b$.
\end{proof}

\section{Projections, elations and the Moufang property}

\subsection{Projections}\label{notationprojection} 

Let $\Delta$ be a thick icosahedral geometry. 
Our next goal is to define what it means to \emph{project} in $\Delta$.

Let $p$ and $b$ be two opposite points in $\Delta$. 
We define $\proj^{b}_{p}$ (or $\proj_{p}$ for short if it is clear from where we are projecting) as follows.

\textbf{Projecting a plane through $b$ to the point $p$:}
Let $\pi$ be a plane through $b$. There exists a unique line in $\pi$ for which every point has distance $2$ from $p$. 
There exists a unique point $p'$ in $p^{\perp}$ that is collinear with that line. We set $\proj_{p} (\pi) := pp'$.

\textbf{Projecting a line through $b$ to the point $p$:}
Let $L$ be a line through $b$. There exists a unique point $\ell$ on $L$ that has distance $2$ from $p$ and a unique line $L'$ in $\ell^{\perp} \cap p^{\perp}$. 
We set $\proj_{p} (L) := \langle p, L' \rangle$.

\begin{nota}
If we are given points $p_{i}$ such that $p_{i}$ is opposite both $p_{i-1}$ and $p_{i+1}$ for $i \in \Z / n \Z$, and we project from $p_0$ to $p_1$, $\dots$, to $p_{n-1}$, then we also write $p_0 \per p_1 \per \dots \per p_{n-1}$. 
If we want to see how a specific point $x$ of $p_{0}^{\perp}$ maps from $p_0$ to $p_1$, $\dots$, to $p_{n-1}$, and it is clear that we are considering $\theta \colon p_0 \per p_1 \per \dots \per p_{n-1}$, then we also write $p_{0}x \mapsto  \proj_{p_{1}} (p_{0}x) \mapsto \dots \mapsto x^{\theta}$. 
\end{nota}

\subsection{Root elations and the Moufang property}

\begin{defn}
An \emph{automorphism} or a \emph{collineation} of a point-line geometry is a permutation of the point set inducing a
permutation of the line set; that is, the set of points on any line is mapped onto the set of points of another line, and each line occurs as such an image. In particular, collineations preserve collinearity. The set of all collineations forms a group called the \emph{collineation group} (cf. \cite[\S 7]{MonographPolarSpaces}).
\end{defn}

\begin{defn}
Let $\Delta$ be a generalized pentagon, and let $\Gamma$ be an apartment in $\Delta$. 

Let $L_0$ and $L_1$ be two intersecting lines of $\Gamma$, and let $c$ be their intersection point.
Let $p$ be the other point of $L_0$, and let $q$ be the other point of $L_1$.
Let $L_2$ be another line of $\Gamma$ through $q$.
Then we call the set $\{p, L_0, c, L_1, q, L_2\}$ a \emph{half-apartment of $\Gamma$},
and denote it by $\alpha$.
We call the set $\alpha^+ := \{L_0, c, L_1, q\}$ 
the \emph{interior of the half-apartment $\alpha$}.
\end{defn}

\begin{defn}
A \emph{root elation of $\Delta$ with respect to a half-apartment $\alpha$} (of some apartment $\Gamma$ of $\Delta$)
is a collineation of $\Delta$ that fixes each point on each line, and each line through each point of $\alpha^+$.
\end{defn}

\begin{defn}
If for every half-apartment $\alpha$ (of any apartment $\Gamma$ of $\Delta$) 
the group generated by all root elations of $\Delta$ with respect to $\alpha$
acts transitively on the set of apartments containing $\alpha$, then we say that $\Delta$ is \emph{Moufang}.
\end{defn}


\section{Constructions of elations}\label{constructionofelationsforH3}

Now, we would like to give constructions for nontrivial root elations for the thick generalized pentagons contained in thick icosahedral geometries. 
The goal is to prove the following proposition.

\begin{prop}\label{elationsH3}
Let $\Delta$ be a thick icosahedral geometry. Then its point residues are Moufang.
\end{prop}

More precisely, we will show the following:
Let $p$ and $b$ be two opposite points of $\Delta$. 
Let $\Sigma$ be an apartment of the generalized pentagon $b^{\perp} \cap p^{\pperp}$ with points 
$$\{b_0, b_1, b_2, b_3, b_4\},$$ 
such that $b_i \perp b_{i+1}$ for $i \in \Z / 5\Z$. 
Then we can construct an elation that fixes the lines $b_0b_1$ and $b_1b_2$ pointwise, 
stabilizes all lines through $b_0$ and $b_1$ in $b^{\perp} \cap p^{\pperp}$, 
and maps $\Sigma$ to an apartment $\Sigma'$ with points $\{b_0, b_1, b_2, b_3', b_4'\}$, 
such that $b_2 \perp b_3' \perp b_4' \perp b_0$, $b_3 \neq b_3'$, and $b_4 \neq b_4'$. 
Since $p$, $b$, $\Sigma$, and $\Sigma'$ were arbitrary, this shows that point residues of $\Delta$ are Moufang.
 
\begin{proof}
Let $p_{0}, \dots, p_{4}$ be points that span a pentagon in $p^{\perp} \cap b^{\pperp}$, such that $p_{i}$ is collinear with $b_{i-1}b_{i}$ for $i \in \Z / 5\Z$. Let $b_4'$ be a point on $b_0b_4$ not equal to $b_0$ or $b_4$, and let $b_3'$ be a point in $b^{\perp} \cap p^{\pperp}$ that is collinear with $b_2$ and $b_4'$. We want to define a map $\theta\colon \Res(b) \rightarrow \Res(b)$ such that $\theta$ fixes $b_0b_1$ and $b_1b_2$ pointwise, stabilizes all lines through $b_0$ and $b_1$, and moves $b_3$ to $b_3'$ and $b_4$ to $b_4'$. Let $d$ be an arbitrary point on $bb_1$ not equal to $b$ or $b_1$.

The point $d$ has distance $2$ to all points of a unique line through $p_0$ in $\<p, p_0, p_4 \>$ that intersects $pp_4$ in some point $p_4'$. Every point on $bb_1$ has distance $2$ to $p_3$. With that, $d$ has distance $2$ to all points of the line $p_3p_4'$ in $\<p, p_3, p_4\>$. We define $d_4 := d^{\perp} \cap (p_0p_4')^{\perp}$ and $d_3 := d^{\perp} \cap (p_3p_4')^{\perp}$. For a visualization, see Figure \ref{1H3}.

We consider the projectivity $\varphi\colon b \per p \per d$ and observe the following, using the notation introduced in \ref{notationprojection}:

\begin{minipage}{0.55\textwidth}
\begin{align*}
bb_2 \mapsto \< p, p_2, p_3 \> &\mapsto db_2 \\
\< b, b_1, b_2 \> \mapsto pp_2 &\mapsto \< d, b_1, b_2\> \\
&= \< b, b_1, b_2 \>  \\
bb_1 \mapsto \< p, p_1, p_2 \> &\mapsto db_1 \\
\< b, b_0, b_1 \> \mapsto pp_1 &\mapsto \< d, b_0, b_1\> \\ 
&= \< b, b_0, b_1 \>  \\
bb_0 \mapsto \< p, p_0, p_1 \> &\mapsto db_0 \\
\< b, b_0, b_4 \> \mapsto pp_0 &\mapsto \< d, b_0, d_4 \>  \\
bb_4 \mapsto \< p, p_0, p_4 \> &\mapsto dd_4 \\
\< b, b_3, b_4 \> \mapsto pp_4 &\mapsto \< d, d_3, d_4 \>  \\
bb_3 \mapsto \< p, p_3, p_4 \> &\mapsto dd_3  \\
\< b, b_2, b_3 \> \mapsto pp_3 &\mapsto \< d, b_2, d_3 \>  
\end{align*}
\end{minipage}
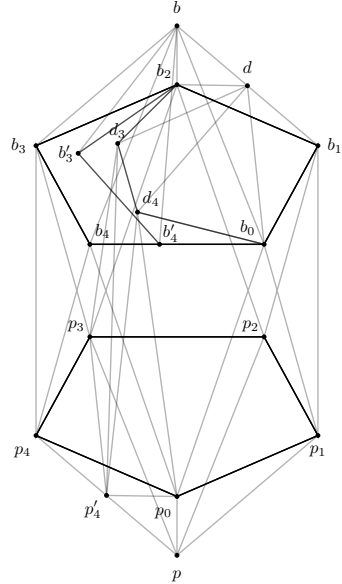
\begin{minipage}[c][7.5cm][c]{0.35\textwidth}
\begin{center}
    \resizebox{4.5cm}{!}{\begin{tikzpicture}[baseline={(0, 3.2)}, scale=0.8]
        \def\pentagonradius{3} 
        \def\angle{72} 
        \def\rotationP{-90} 
        \def\rotationB{-54} 
        \def\verticalshift{8} 

        \begin{scope}[scale=1, yscale=0.6]
            \foreach \i in {0,1,2,3,4} {
                \coordinate (P\i) at (\i*\angle+\rotationP:\pentagonradius);
            }
            \draw[thick] (P0) -- (P1) -- (P2) -- (P3) -- (P4) -- cycle;

            \node at (P0) [circle, fill, inner sep=1pt, label={[label distance=-1.5mm, yshift=1mm]below right:{\small \( p_0 \)}}] {};
            \node at (P1) [circle, fill, inner sep=1pt, label={[label distance=-1mm]right:{\small \( p_1 \)}}] {};
            \node at (P2) [circle, fill, inner sep=1pt, label={[label distance=-1mm]above left:{\small \( p_2 \)}}] {};
            \node at (P3) [circle, fill, inner sep=1pt, label={[label distance=-1mm]above right:{\small \( p_3 \)}}] {};
            \node at (P4) [circle, fill, inner sep=1pt, label={[label distance=-1mm]left:{\small \( p_4 \)}}] {};        

            \draw[thick] (P0) -- (P1);
            \draw[thick] (P1) -- (P2);
            \draw[thick] (P2) -- (P3);
            \draw[thick] (P3) -- (P4);
            \draw[thick] (P4) -- (P0);

            \coordinate (P) at (0,-5);
            \node at (P) [circle, fill, inner sep=1pt, label={[label distance=-1mm]below:{\small \( p \)}}] {};
        \end{scope}

        \begin{scope}[scale=1, yscale=0.6, shift={(0,\verticalshift)}]
            \foreach \i in {0,1,2,3,4} {
                \coordinate (B\i) at (\i*\angle+\rotationB:\pentagonradius);
            }
            \draw[thick] (B0) -- (B1) -- (B2) -- (B3) -- (B4) -- cycle;

            \node at (B0) [circle, fill, inner sep=1pt, label={[label distance=-1mm]above left:{\small \( b_0 \)}}] {};
            \node at (B1) [circle, fill, inner sep=1pt, label={[label distance=0mm]right:{\small \( b_1 \)}}] {};
            \node at (B2) [circle, fill, inner sep=1pt, label={[label distance=-1mm]above left:{\small \( b_2 \)}}] {};
            \node at (B3) [circle, fill, inner sep=1pt, label={[label distance=0mm]left:{\small \( b_3 \)}}] {};
            \node at (B4) [circle, fill, inner sep=1pt, label={[label distance=-1mm] above right:{\small \( b_4 \)}}] {};    

            \draw[thick] (B0) -- (B1);
            \draw[thick] (B1) -- (B2);
            \draw[thick] (B2) -- (B3);
            \draw[thick] (B3) -- (B4);
            \draw[thick] (B4) -- (B0);

            \coordinate (B) at (0,5);
            \node at (B) [circle, fill, inner sep=1pt, label={[label distance=-1mm]above:{\small \( b \)}}] {};

            \coordinate (D) at ($ (B) + 0.5*(B1) - 0.5*(B) $); 
            \node at (D) [circle, fill, inner sep=1pt, label={[label distance=-1mm]above right:{\small \( d \)}}] {};    
        \end{scope}

        \draw[thick, opacity=0.3] (P0) -- (B0) -- (P1) -- (B1) -- (P2) -- (B2)
                                 -- (P3) -- (B3) -- (P4) -- (B4) -- (P0);

        \draw[thick, opacity=0.3] (P0) -- (P); 
        \draw[thick, opacity=0.3] (P1) -- (P); 
        \draw[thick, opacity=0.3] (P2) -- (P); 
        \draw[thick, opacity=0.3] (P3) -- (P); 
        \draw[thick, opacity=0.3] (P4) -- (P);

        \draw[thick, opacity=0.3] (B0) -- (B); 
        \draw[thick, opacity=0.3] (B1) -- (B); 
        \draw[thick, opacity=0.3] (B2) -- (B); 
        \draw[thick, opacity=0.3] (B3) -- (B); 
        \draw[thick, opacity=0.3] (B4) -- (B);
 
        \draw[thick, opacity=0.3] (B0) -- (D) -- (B2);
    

    
        \coordinate (D4) at (-0.8, 4.0); 
        \node at (D4) [circle, fill, inner sep=1pt, label={[label distance=-1mm]above right:{\small \( d_4 \)}}] {};
        \draw[thick, opacity=0.3] (B0) -- (D4) -- (D);
    
        \coordinate (P4') at ($ (P) + 0.5*(P4) - 0.5*(P) $); 
        \node at (P4') [circle, fill, inner sep=1pt, label={[label distance=-1mm]below left:{\small \( p_{4}' \)}}] {};
        \draw[thick, opacity=0.3] (P0) -- (P4') -- (P3);    
        \draw[thick, opacity=0.3] (P0) -- (D4) -- (P4');    
    
    
        \coordinate (D3) at (-1.2, 5.4); 
        \node at (D3) [circle, fill, inner sep=1pt, label={[label distance=-1mm]right:{\small \( d_3 \)}}] {};
        \draw[thick, opacity=0.7] (B2) -- (D3) -- (D4);    
        \draw[thick, opacity=0.7] (B0) -- (D4);   
        \draw[thick, opacity=0.3] (D3) -- (D);    
        \draw[thick, opacity=0.3] (P3) -- (D3) -- (P4');    

        \coordinate (B3') at (-2,5.2);
        \node at (B3') [circle, fill, inner sep=1pt, label={[label distance=-1mm]left:{\small \( b_{3}' \)}}] {};    
    
        \coordinate (B4') at ($ (B4) + 0.4*(B0) - 0.4*(B4) $); 
        \node at (B4') [circle, fill, inner sep=1pt, label={[label distance=-1.5mm]above right:{\small \( b_{4}' \)}}] {};    
    
        \draw[thick, opacity=0.7] (B4') -- (B3') -- (B2);    
        \draw[thick, opacity=0.3] (B4') -- (B) -- (B3');          
      
\end{tikzpicture}}
\end{center}
\end{minipage}
\begin{flushright}
\captionof{figure}{The points and lines relevant for the construction of a root elation for the generalized pentagon $b^{\perp} \cap p^{\pperp}$.}
\end{flushright}
\label{1H3}

Let $x$ be an arbitrary point on $b_0b_1$. Then $bx$ maps to a plane through $pp_1$ and then to $dx$. Let $y$ be an arbitrary point on $b_1b_2$. Then $by$ maps to a plane through $pp_2$ and then to $dy$.

We observe that the points $b_3'$ and $d_4$ are opposite and consider an apartment $\Sigma'$ containing $b_3'$, $b$, $b_2$, $b_4'$, $b_0$, $d$, $d_3$, and $d_4$. The points $b_4'$ and $d_4$ have distance $2$ and are both collinear with $b_0$. Let $q_0$ be another point on the line $(b_4')^{\perp} \cap d_4^{\perp}$ through $b_0$. We label the leftover point of the pentagon $d_4^{\perp} \cap (b_3')^{\pperp}$ in $\Sigma'$ by $q_4$. Let $q_4'$ be the point in $\Sigma'$ collinear with $q_0$, $q_4$, $b_3'$, and $b_4'$, and let $q_3$ be the point collinear with $q_4'$, $q_4$, $b_2$, $b_3'$, and $d_3'$. With that, all points of $\Sigma'$ are labeled. The point $b_1$ is on the line $bd$. Let $q$ be the unique point on $q_4q_4'$ with distance $2$ to $b_1$. 

We see that $q$ has distance $3$ to $b$ and $d$, distance $2$ to $b_0$, $b_1$, $b_2$, $b_3'$, $b_4'$, $d_3$, and $d_4$, and distance $1$ to $q_0$, $q_4$, $q_4'$, $q_3$, $q_2$, and $q_1$. We consider an apartment containing $\<q, q_0, q_4'\>$ and $\<b, b_1, b_2\>$. We denote the point in that apartment that is collinear with $q$, $q_0$, $b_0$, and $b_1$ by $q_1$, and the point that is collinear with $q$, $q_3$, $b_1$, and $b_2$ by $q_2$ (cf.~Figure \ref{2H31},~Figure \ref{2H32}).

\begin{center}
\begin{tabular}{
    >{\centering\arraybackslash} m{4.5cm} 
    >{\centering\arraybackslash} m{4.5cm}
}

\resizebox{4cm}{!}{
\begin{tikzpicture}[baseline={(0, 3.2)}, scale=0.8]
    
    \def\pentagonradius{3}
    \def\angle{72} 
    \def\rotationP{-90} 
    \def\rotationB{-54} 
    \def\verticalshift{8} 

    \begin{scope}[scale=1, yscale=0.6]
        \foreach \i in {0,1,2,3,4} {
            \coordinate (Q\i) at (\i*\angle+\rotationP:\pentagonradius);
        }
        \draw[thick] (Q0) -- (Q1) -- (Q2) -- (Q3) -- (Q4) -- cycle;

        \node at (Q0) [circle, fill, inner sep=1pt, label={[label distance=-1.5mm]below right:\( q_4 \)}] {};
        \node at (Q1) [circle, fill, inner sep=1pt, label={[label distance=-1mm]right:\( q_0 \)}] {};
        \node at (Q2) [circle, fill, inner sep=1pt, label={[label distance=-1mm, yshift=-1mm]above right:\( b_0 \)}] {};
        \node at (Q3) [circle, fill, inner sep=1pt, label={[label distance=-1mm, yshift=-1mm]above left:\( d \)}] {};
        \node at (Q4) [circle, fill, inner sep=1pt, label={[label distance=-1mm]left:\( d_3 \)}] {};

        \coordinate (Q) at (0,-5);
        \node at (Q) [circle, fill, inner sep=1pt, label={below:\( d_4 \)}] {};
    \end{scope}

    \begin{scope}[scale=1, yscale=0.6, shift={(0,\verticalshift)}]
        \foreach \i in {0,1,2,3,4} {
            \coordinate (B\i) at (\i*\angle+\rotationB:\pentagonradius);
        }
        \draw[thick] (B0) -- (B1) -- (B2) -- (B3) -- (B4) -- (B0);

        \node at (B0) [circle, fill, inner sep=1pt, label={[label distance=-1mm]below right:\( q_4' \)}] {};
        \node at (B1) [circle, fill, inner sep=1pt, label={[label distance=-1mm]right:\( b_4' \)}] {};            
        \node at (B2) [circle, fill, inner sep=1pt, label={[label distance=-1.5mm]above left:\( b \)}] {};
        \node at (B3) [circle, fill, inner sep=1pt, label={[label distance=-1mm]left:\( b_2 \)}] {};
        \node at (B4) [circle, fill, inner sep=1pt, label={[label distance=-1mm]below left:\( q_3 \)}] {};            

        \coordinate (D) at (0,5);
        \node at (D) [circle, fill, inner sep=1pt, label={above:\( b_3' \)}] {};
    \end{scope}

    \draw[thick, opacity=0.3] 
        (Q0) -- (B0) -- (Q1) -- (B1) -- (Q2) -- (B2) 
        -- (Q3) -- (B3) -- (Q4) -- (B4) -- (Q0);

    \foreach \i in {0,1,2,3,4} {
        \draw[thick, opacity=0.3] (Q\i) -- (Q);
        \draw[thick, opacity=0.3] (B\i) -- (D);
    }

\end{tikzpicture}}
\captionof{figure}{The apartment containing $\<b_3', b, b_2\>$ and $\<d_4, q_0, q_4\>$.}
\label{2H31}

&   

\resizebox{4cm}{!}{
\begin{tikzpicture}[baseline={(0, 3.2)}, scale=0.8]

    \def\pentagonradius{3}
    \def\angle{72} 
    \def\rotationP{-90} 
    \def\rotationB{-54} 
    \def\verticalshift{8} 

    \begin{scope}[scale=1, yscale=0.6]
        \foreach \i in {0,1,2,3,4} {
            \coordinate (Q\i) at (\i*\angle+\rotationP:\pentagonradius);
        }
        \draw[thick] (Q0) -- (Q1) -- (Q2) -- (Q3) -- (Q4) -- cycle;

        \node at (Q0) [circle, fill, inner sep=1pt, label={[label distance=-1mm]below right:\( q_0 \)}] {};
        \node at (Q1) [circle, fill, inner sep=1pt, label={[label distance=-1mm]right:\( q_1 \)}] {};
        \node at (Q2) [circle, fill, inner sep=1pt, label={[label distance=-1mm, yshift=-1mm]above right:\( q_2 \)}] {};
        \node at (Q3) [circle, fill, inner sep=1pt, label={[label distance=-1mm, yshift=-1mm]above left:\( q_3 \)}] {};
        \node at (Q4) [circle, fill, inner sep=1pt, label={[label distance=-1mm]left:\( q_4' \)}] {};

        \coordinate (Q) at (0,-5);
        \node at (Q) [circle, fill, inner sep=1pt, label={[label distance=-1mm]below:\( q \)}] {};
    \end{scope}

    \begin{scope}[scale=1, yscale=0.6, shift={(0,\verticalshift)}]
        \foreach \i in {0,1,2,3,4} {
            \coordinate (B\i) at (\i*\angle+\rotationB:\pentagonradius);
        }
        \draw[thick] (B0) -- (B1) -- (B2) -- (B3) -- (B4) -- (B0);

        \node at (B0) [circle, fill, inner sep=1pt, label={[label distance=-1mm]below right:\( b_0 \)}] {};
        \node at (B1) [circle, fill, inner sep=1pt, label={[label distance=-1mm]right:\( b_1 \)}] {};            
        \node at (B2) [circle, fill, inner sep=1pt, label={[label distance=-1.5mm]above left:\( b_2 \)}] {};
        \node at (B3) [circle, fill, inner sep=1pt, label={[label distance=-1mm]left:\( b_3' \)}] {};
        \node at (B4) [circle, fill, inner sep=1pt, label={[label distance=-1mm]below left:\( b_4' \)}] {};         

        \coordinate (D) at (0,5);
        \node at (D) [circle, fill, inner sep=1pt, label={[label distance=-1mm]above:\( b \)}] {};
    \end{scope}

    \draw[thick, opacity=0.3] 
        (Q0) -- (B0) -- (Q1) -- (B1) -- (Q2) -- (B2) 
        -- (Q3) -- (B3) -- (Q4) -- (B4) -- (Q0);

    \foreach \i in {0,1,2,3,4} {
        \draw[thick, opacity=0.3] (Q\i) -- (Q);
        \draw[thick, opacity=0.3] (B\i) -- (D);
    }

\end{tikzpicture}}
\captionof{figure}{The apartment containing $\<b, b_1, b_2\>$ and $\<q, q_0, q_4'\>$.}
\label{2H32}

\end{tabular}
\end{center}

We set $\theta := b \per p \per d \per q \per b$ and claim that $\theta$ has the desired properties. Indeed, we observe the following.

\begin{align*}
bb_2 \mapsto \< p, p_2, p_3 \> &\mapsto db_2 \mapsto \<q, q_2, q_3\> \mapsto bb_2\\
\< b, b_1, b_2 \> \mapsto pp_2 &\mapsto \< d, b_1, b_2\> = \< b, b_1, b_2 \>  \mapsto qq_2 \mapsto \< b, b_1, b_2 \> \\
bb_1 \mapsto \< p, p_1, p_2 \> &\mapsto db_1 \mapsto \<q, q_1, q_2\> \mapsto bb_1\\
\< b, b_0, b_1 \> \mapsto pp_1 &\mapsto \< d, b_0, b_1\> = \< b, b_0, b_1 \>  \mapsto qq_1 \mapsto \<b, b_0, b_1\>\\
bb_0 \mapsto \< p, p_0, p_1 \> &\mapsto db_0 \mapsto \<q, q_0, q_1\> \mapsto bb_0\\
\< b, b_0, b_4 \> \mapsto pp_0 &\mapsto \< d, b_0, d_4 \> \mapsto qp_0 \mapsto \< b, b_0, b_4 \> \\
bb_4 \mapsto \< p, p_0, p_4 \> &\mapsto dd_4 \mapsto \< q, p_0, q_4 \> \mapsto bb_4'\\
\< b, b_3, b_4 \> \mapsto pp_4 &\mapsto \< d, d_3, d_4 \> \mapsto qq_4 \mapsto \< b, b_3', b_4' \> \\
bb_3 \mapsto \< p, p_3, p_4 \> &\mapsto dd_3 \mapsto \< q, q_3, q_4 \> \mapsto bb_3'  \\
\< b, b_2, b_3 \> \mapsto pp_3 &\mapsto \< d, b_2, d_3 \> \mapsto qq_3 \mapsto \< b, b_2, b_3'\>  
\end{align*}

Let $x$ be some arbitrary point on $b_0b_1$. Then $bx$ maps to a plane through $pp_1$, to $dx$, to a plane through $qp_1$ and back to $bx$.

Let $y$ be some arbitrary point on $b_1b_2$. Then $by$ maps to a plane through $pp_2$, to $dy$, to a plane through $qq_2$ and back to $by$.

To see this in a drawing, we refer to figure \ref{3H3}.

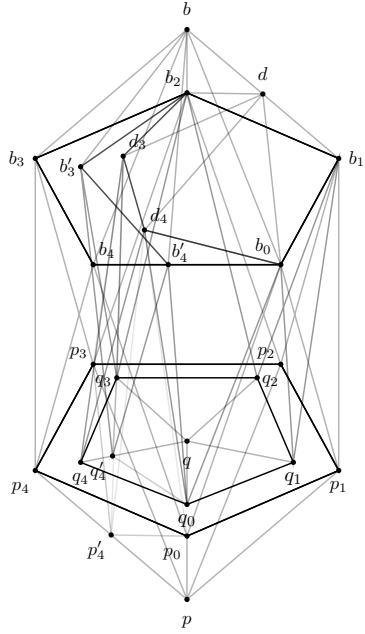
\begin{center}
\resizebox{5cm}{!}{
\begin{tikzpicture}[baseline={(0, 3.2)}, scale=0.8]
    
    \def\pentagonradius{3} 
    \def\angle{72}
    \def\rotationP{-90} 
    \def\rotationB{-54} 
    \def\verticalshift{8} 

    \begin{scope}[scale=1, yscale=0.6]

        \foreach \i in {0,1,2,3,4} {
            \coordinate (P\i) at (\i*\angle+\rotationP:\pentagonradius);
        }
        \draw[thick] (P0) -- (P1) -- (P2) -- (P3) -- (P4) -- cycle;

        \node at (P0) [circle, fill, inner sep=1pt, label={[label distance=-1mm]below left:{\small \( p_0 \)}}] {};
        \node at (P1) [circle, fill, inner sep=1pt, label={[label distance=-1mm]right:{\small \( p_1 \)}}] {};
        \node at (P2) [circle, fill, inner sep=1pt, label={[label distance=-1mm]above right:{\small \( p_2 \)}}] {};
        \node at (P3) [circle, fill, inner sep=1pt, label={[label distance=-1mm]above left:{\small \( p_3 \)}}] {};
        \node at (P4) [circle, fill, inner sep=1pt, label={[label distance=-1mm]left:{\small \( p_4 \)}}] {};        

        \draw[thick] (P0) -- (P1);
        \draw[thick] (P1) -- (P2);
        \draw[thick] (P2) -- (P3);
        \draw[thick] (P3) -- (P4);
        \draw[thick] (P4) -- (P0);

        \coordinate (P) at (0,-5);
        \node at (P) [circle, fill, inner sep=1pt, label={[label distance=-1mm]below:{\small \( p \)}}] {};

        \coordinate (Q) at (0,0);
        \node at (Q) [circle, fill, inner sep=1pt, label={[label distance=-1mm]below:{\small \( q \)}}] {};

        \coordinate (Q0) at (0, -2);
        \coordinate (Q1) at (2, -0.67);
        \coordinate (Q2) at (1.32, 2);
        \coordinate (Q3) at (-1.32, 2);
        \coordinate (Q4) at (-2, -0.67);

        \draw[thick] (Q0) -- (Q1) -- (Q2) -- (Q3) -- (Q4) -- cycle;

        \node at (Q0) [circle, fill, inner sep=1pt, label={[label distance=-1mm, yshift=0.7mm] below:\(q_0\)}] {};
        \node at (Q1) [circle, fill, inner sep=1pt, label={[label distance=-1mm, yshift=0.7mm] below:\(q_1\)}] {};
        \node at (Q2) [circle, fill, inner sep=1pt, label={[label distance=-1mm] right:{\small \(q_2\)}}] {};
        \node at (Q3) [circle, fill, inner sep=1pt, label={[label distance=-1mm] left:{\small \(q_3\)}}] {};
        \node at (Q4) [circle, fill, inner sep=1pt, label={[label distance=-1mm, yshift=0.7mm] below:{\small \(q_4\)}}] {};
           
    \end{scope}

    \begin{scope}[scale=1, yscale=0.6, shift={(0,\verticalshift)}]

        \foreach \i in {0,1,2,3,4} {
            \coordinate (B\i) at (\i*\angle+\rotationB:\pentagonradius);
        }
        \draw[thick] (B0) -- (B1) -- (B2) -- (B3) -- (B4) -- cycle;

        \node at (B0) [circle, fill, inner sep=1pt, label={[label distance=-1mm]above left:{\small \( b_0 \)}}] {};
        \node at (B1) [circle, fill, inner sep=1pt, label={[label distance=-1mm]right:{\small \( b_1 \)}}] {};
        \node at (B2) [circle, fill, inner sep=1pt, label={[label distance=-1.5mm]above left:{\small \( b_2 \)}}] {};
        \node at (B3) [circle, fill, inner sep=1pt, label={[label distance=-1mm]left:{\small \( b_3 \)}}] {};
        \node at (B4) [circle, fill, inner sep=1pt, label={[label distance=-1mm] above right:{\small \( b_4 \)}}] {};    

        \draw[thick] (B0) -- (B1);
        \draw[thick] (B1) -- (B2);
        \draw[thick] (B2) -- (B3);
        \draw[thick] (B3) -- (B4);
        \draw[thick] (B4) -- (B0);

        \coordinate (B) at (0,5);
        \node at (B) [circle, fill, inner sep=1pt, label={[label distance=-1mm]above:{\small \( b \)}}] {};

        \coordinate (D) at ($ (B) + 0.5*(B1) - 0.5*(B) $); 
        \node at (D) [circle, fill, inner sep=1pt, label={[label distance=-1mm]above right:{\small \( d \)}}] {};    
     
    \end{scope}

    \draw[thick, opacity=0.3] (P0) -- (B0) -- (P1) -- (B1) -- (P2) -- (B2) -- (P3) -- (B3) -- (P4) -- (B4) -- (P0);

    \foreach \i in {0,1,2,3,4} {
        \draw[thick, opacity=0.3] (P\i) -- (P);
        \draw[thick, opacity=0.3] (B\i) -- (B);
        \draw[thick, opacity=0.3] (Q\i) -- (Q);
    }

    \draw[thick, opacity=0.3] (B0) -- (D) -- (B2);

    \coordinate (D4) at (-0.8, 4.0); 
    \node at (D4) [circle, fill, inner sep=1pt, label={[label distance=-1mm]above right:{\small \( d_4 \)}}] {};
    \draw[thick, opacity=0.3] (B0) -- (D4) -- (D);

    \coordinate (P4') at ($ (P) + 0.5*(P4) - 0.5*(P) $); 
    \node at (P4') [circle, fill, inner sep=1pt, label={[label distance=-1mm]below left:{\small \( p_{4}' \)}}] {};
    \draw[thick, opacity=0.2] (P0) -- (P4') -- (P3);    
    \draw[thick, opacity=0.1] (P0) -- (D4) -- (P4');    

    \coordinate (D3) at (-1.2, 5.4); 
    \node at (D3) [circle, fill, inner sep=1pt, label={[label distance=-1mm, yshift=1mm]below right:{\small \( d_3 \)}}] {};
    \draw[thick, opacity=0.7] (B2) -- (D3) -- (D4);    
    \draw[thick, opacity=0.7] (B0) -- (D4);   
    \draw[thick, opacity=0.3] (D3) -- (D);    
    \draw[thick, opacity=0.1] (P3) -- (D3) -- (P4');    

    \coordinate (B3') at (-2,5.2);
    \node at (B3') [circle, fill, inner sep=1pt, label={[label distance=-1mm]left:{\small \( b_{3}' \)}}] {};    
    \coordinate (B4') at ($ (B4) + 0.4*(B0) - 0.4*(B4) $); 
    \node at (B4') [circle, fill, inner sep=1pt, label={[label distance=-1.5mm]above right:{\small \( b_{4}' \)}}] {};    
    \draw[thick, opacity=0.7] (B4') -- (B3') -- (B2);    
    \draw[thick, opacity=0.3] (B4') -- (B) -- (B3');          

    \coordinate (Q4') at ($ (Q4) + 0.3*(Q) - 0.3*(Q4) $); 
    \node at (Q4') [circle, fill, inner sep=1pt, label={[label distance=-2.5mm, xshift=0.7mm]below right:{\small \( q_{4}' \)}}] {};     
    \draw[thick, opacity=0.2] (Q3) -- (Q4') -- (Q0);            

    \draw[thick, opacity=0.4, color=black] 
        (Q0) -- (B0) -- (Q1) -- (B1) -- (Q2) -- (B2) -- (Q3) -- (B3') -- (Q4') -- (B4') -- (Q0);  
    \draw[thick, opacity=0.4, color=black] 
        (Q0) -- (D4) -- (Q4) -- (D3) -- (Q3);
    
\end{tikzpicture}
}
\captionof{figure}{The established points and lines relevant to determining the images under $\theta$.}
\label{3H3}
\end{center}

The last things we have to show are that all lines through $b_0$ and $b_1$ in $b^{\perp} \cap p^{\pperp}$ are stabilized under $\theta$. For that, we will first consider an arbitrary plane $\beta$ through $bb_1$. Since $bb_1$ first maps to $\<p, p_1, p_2\>$, $\beta$ will first map to a line in $\<p, p_1, p_2\>$. Since $\proj_{d} (\<p, p_1, p_2\>) = db_1$ and $d \in \beta$, this line will map back to $\beta$. Similarly, $\proj_q (\beta)$ is a line in $\<q, q_1, q_2\>$ that will map back to $\beta$ when we project to $b$. This shows that all lines through $b_1$ in $b^{\perp} \cap p^{\pperp}$ are stabilized under $\theta$.

For the lines through $b_0$, we first consider the projectivity 
$$\psi\colon p \per b \per q \per d \per p.$$ 
We see that $p_1p_2$ is fixed pointwise under $\psi$. Furthermore, a plane $\gamma$ through $pp_1$ maps to a line through $b$ in $\<b, b_0, b_1\>$ that intersects $b_0b_1$ in some point $g$, to a plane through $qq_1$, to a line through $d$ in $\<b, b_0, b_1\>$ that intersects $b_0b_1$ in the same point $g$, and back to $\gamma$. Thus, planes through $pp_1$ are stabilized under $\psi$, and in the same way, we can see that planes through $pp_2$ are stabilized. 

We can define a projectivity $\theta' := p \per b \per q' \per d' \per p$, where $q'$ is a point on $pp_1$ and $d'$ is defined analogously to $q$ previously, opposite $p$ and $d'$, and such that $(\proj_{p}^{b})^{-1} \circ \theta' \circ \proj_{p}^{b}$ moves $b_3'$ to $b_3$. Because of the given symmetry, $\theta'$ will fix $p_0p_1$ and $p_1p_2$ pointwise and stabilize all planes through $pp_1$. 

Now we can consider 
$$\psi'\colon b \per p \per d' \per q' \per b$$ 
and see that $\psi'$ fixes $b_0b_1$ pointwise and stabilizes all planes through $bb_0$ and $bb_1$.

We can see that
$$\psi' \circ \theta^{-1} = b \per q \per d \per p \per b \per p \per d' \per q' \per b$$
fixes all planes through $bb_1$ and all points on $b_0b_1$, since both $\psi'$ and $\theta^{-1}$ do that. But then $\psi' \circ \theta^{-1}$ fixes an apartment in $b^{\perp}$ and therefore has to be the identity. Thus, $\theta$ is a root elation.

\end{proof}

\begin{coro}\label{constructionofelationsforH3corollary}
Thick icosahedral geometries do not exist.
\end{coro}

\begin{proof}
The assertion now follows directly from \cite[Théorème 1]{Tits76}.
\end{proof}


\section{A note on buildings}\label{NoteOnBuildings}

In 1974, Jacques Tits observed that all thick irreducible spherical buildings $\Delta$ of rank at least three have the Moufang property (cf.~\cite{SphericalTypeFiniteBNPairs}); a proof was later provided in \cite{Tits77}.
This groundbreaking result laid the basis for classifying all such buildings, which is sometimes described as one of Tits's greatest achievements (cf.~\cite[\S~0.5]{ABBuildings}) or as one of the great accomplishments of 20th-century mathematics in the literature (cf.~\cite[preface]{StructureSpherical}). 
As part of this research, Jacques Tits showed that Moufang buildings of types $\mathsf{H_3}$ and $\mathsf{H_4}$ do not exist (cf.~\cite{Tits77}).
This is equivalent to showing that thick buildings of types $\mathsf{H_3}$ and $\mathsf{H_4}$ do not exist.
For both showing the Moufang property and the nonexistence of thick buildings of types $\mathsf{H_3}$ and $\mathsf{H_4}$, Jacques Tits used another theorem of his about the existence of extensions of isometries (cf.~\cite[Theorem~4.1.2, Proposition~4.16]{SphericalTypeFiniteBNPairs}).
As mentioned in the introduction, the main goal of this article was to provide a new, elementary proof of the nonexistence of thick buildings of types $\mathsf{H_3}$ and $\mathsf{H_4}$ that avoids the use of the extension theorem.

Recall that the Coxeter diagram for $\mathsf{H_3}$ is the following: 
\begin{center}
\begin{tikzpicture}
\begin{scope}[xscale=1.3, yscale=1.3]
\coordinate (1) at (0,0); 
\coordinate (2) at (1, 0);  
\coordinate (3) at (2, 0); 

\draw[] (1) -- (2) -- (3);

\coordinate (caption) at (1.5,0); 
\end{scope}
\node at (caption) [label={[label distance=-2mm]above:{\small \( 5 \)}}] {};

\node at (1) [circle, fill, inner sep=1pt, label={[label distance=-1mm]above right:{\small \(  \)}}] {};
\node at (2) [circle, fill, inner sep=1pt, label={[label distance=-1mm]above:{\small \( \)}}] {};
\node at (3) [circle, fill, inner sep=1pt, label={[label distance=-1mm]left:{\small \(  \)}}] {};
\end{tikzpicture}
\end{center}

Since the proper subdiagrams of this diagram are of types $\mathsf{A_2}$ and $\mathsf{I_2}(5)$, 
we deduce that a Lie incidence geometry $\Delta$ of type $\mathsf{H_3}$ 
has projective planes and generalized pentagons as rank-$2$ residues and icosahedra as apartments 
(for details, see \cite[\S~2.1,~\S~2.4]{DiagramGeometry}). 
Points can coincide or be at distances $1$, $2$, or $3$.
Consequently, one can infer from the diagram that Lie incidence geometries of type $\mathsf{H_3}$ 
satisfy the axioms of icosahedral geometries (cf.~Definition \ref{icosahedralGeometriesDefn}),
and moreover that thick buildings of type $\Hdrei$ correspond to thick icosahedral geometries.
Hence, assuming that thick buildings of type $\Hdrei$ exist, the root elations constructed in the proof of Proposition \ref{elationsH3} 
correspond to root elations of residues of type $\mathsf{I_2}(5)$ in thick buildings of type $\Hdrei$,
leading to the same contradiction as before and showing that thick buildings of type $\Hdrei$ cannot exist (cf.~Corollary \ref{constructionofelationsforH3corollary}).

In \cite{Busch:TheMoufangConditionForPolarSpaces}, the author provided direct geometric constructions for all root elations of 
rank-$2$ residues of buildings of type $\mathsf{B_n} = \mathsf{C_n}$ corresponding to generalized quadrangles and showed that
these extend to the ambient building.
Such constructions were given by Jacques Tits in \cite[Proposition~7.11]{SphericalTypeFiniteBNPairs} for the other rank-$2$ residues
of buildings of type $\mathsf{B_n} = \mathsf{C_n}$, which are projective planes, and it follows from \cite{Busch:TheMoufangConditionForPolarSpaces} that these extend to the ambient building as well.
Considering all types of spherical Coxeter diagrams, we observe that only
$\mathsf{A_3}=\mathsf{D_3}$, $\mathsf{B_3}=\mathsf{C_3}$, and $\mathsf{H_3}$ occur.
Since buildings of type $\mathsf{A_n}$ correspond to projective spaces, they have been studied extensively, 
and geometric proofs for the fact that they are Moufang, if their dimension is at least three, are known 
(see, for example, \cite[Ch.~6.3]{DiagramGeometry}). 
Thus, this article and \cite{Busch:TheMoufangConditionForPolarSpaces}, together with the known constructions for buildings of type $\mathsf{A_3} = \mathsf{D_3}$, provide an alternative elementary proof of the fact that all thick irreducible spherical buildings of rank $3$ are Moufang without using Tits's extension theorem or Tits's classification.


\subsection*{Acknowledgment} The author is grateful to Prof. Linus Kramer for the idea for this project. 
Furthermore, the author wants to thank Prof. Linus Kramer and Prof. Hendrik Van Maldeghem for their general advice and support. 
In particular, the discussions with Prof. Hendrik Van Maldeghem about the subjects of this article have been very insightful. 
The author would also like to thank the anonymous referee for suggesting to find an axiom system for icosahedral geometries 
and to express the contents of this article purely in the language of point-line geometries (except for Section \ref{NoteOnBuildings}).

The author is funded by the Claussen-Simon-Stiftung and by the Deutsche Forschungsgemeinschaft (DFG, German Research Foundation) under Germany's Excellence Strategy EXC 2044 --390685587, Mathematics Münster: Dynamics--Geometry--Structure. This work is part of the PhD project of the author.


\begin{thebibliography}{99}

\bibitem{ABBuildings}
P.~Abramenko \& K.~S.~Brown,
\emph{Buildings: Theory and applications},
Graduate Texts in Mathematics \textbf{248},
Springer, New York, 2008.

\bibitem{Handbook}
F.~Buekenhout, 
\emph{Handbook of Incidence Geometry}, 
Elsevier B.V., 1995.

\bibitem{Busch:TheMoufangConditionForPolarSpaces}
S.~Busch,
\emph{The Moufang condition for polar spaces},
submitted, 2025.

\bibitem{DiagramGeometry}
F.~Buekenhout \& A.~M.~Cohen,
\emph{Diagram geometry related to classical groups and buildings},
A Series of Modern Surveys in Mathematics \textbf{57},
Springer, Heidelberg, 2013.

\bibitem{SphericalTypeFiniteBNPairs}
J.~Tits,
\emph{Buildings of spherical type and finite BN-pairs},
Lecture Notes in Mathematics \textbf{386},
Springer-Verlag, Berlin, 1974 (2nd printing, 1986).

\bibitem{Tits76}
J.~Tits,
\emph{Non-existence de certains polygones généralisés, I},
Invent.\ Math.\ \textbf{36}, 1976, 275--284.

\bibitem{Tits77}
J.~Tits,
\emph{Endliche Spiegelungsgruppen, die als Weylgruppen auftreten},
Invent.\ Math.\ \textbf{43}, 1977, 283--295.

\bibitem{MonographPolarSpaces}
H.~Van~Maldeghem,
\emph{Polar spaces},
Münster Lectures in Mathematics,
Europ.\ Math.\ Soc.\ Press, 2024.

\bibitem{GeneralizedPolygons}
H.~Van~Maldeghem,
\emph{Generalized polygons},
Monographs in Mathematics \textbf{93},
Birkhäuser, 1998.

\bibitem{Mal:95}
H.~Van~Maldeghem,
\emph{The non-existence of certain regular generalized polygons},
Arch.\ Math.\ \textbf{64}, 1995, 86--96.

\bibitem{StructureSpherical}
R.~M.~Weiss,
\emph{The structure of spherical buildings},
Princeton University Press, 2003.

\end{thebibliography}
\end{document}